\documentclass[10pt,a4paper]{amsart}
\usepackage{amsrefs, graphicx, mathrsfs, amssymb, color, longtable}
\usepackage[T1]{fontenc}
\usepackage[utf8]{inputenc}
\usepackage[body={15cm, 24cm}]{geometry}
\usepackage[pagewise]{lineno}

\newtheorem{theorem}{Theorem}[section]
\newtheorem{proposition}[theorem]{Proposition}
\newtheorem{lemma}[theorem]{Lemma}
\newtheorem{corollary}[theorem]{Corollary}

\theoremstyle{definition}
\newtheorem{definition}[theorem]{Definition}
\newtheorem{hypothesis}[theorem]{Hypothesis}

\theoremstyle{remark}

\newtheorem{example}[theorem]{Example}
\newtheorem{remark}[theorem]{Remark}

\numberwithin{equation}{section}

\title[Attainability in Wasserstein spaces]{Attainability property for a probabilistic target in Wasserstein spaces}

\author[G. Cavagnari]{Giulia Cavagnari}
\address{\hspace{-0.5em}\begin{tabular}{ll}Giulia Cavagnari:&Politecnico di Milano,\\& Dipartimento di Matematica ``F. Brioschi''\\& 
Piazza Leonardo da Vinci 32, \\& I-20133 Milano, Italy.\end{tabular}}
\email{giulia.cavagnari@polimi.it}

\author[A. Marigonda]{Antonio Marigonda}
\address{\hspace{-0.5em}\begin{tabular}{ll}Antonio Marigonda:&Department of Computer Science,\\& University of Verona\\ &Strada Le Grazie 15, I-37134 Verona, Italy.\end{tabular}}
\email{antonio.marigonda@univr.it}

\date{\today}



\begin{document}

\begin{abstract}
In this paper we establish an attainability result for the minimum time function
of a control problem in the space of probability measures endowed with Wasserstein distance.
The dynamics is provided by a suitable controlled continuity equation, 
where we impose a nonlocal nonholonomic constraint on the driving vector field, 
which is assumed to be a Borel selection of a given set-valued map.
This model can be used to describe at a macroscopic level a so-called \emph{multiagent system}
made of several possible interacting agents.
\end{abstract}

\maketitle

\section{Introduction}\label{sec:intro}
We consider a finite-dimensional \emph{multiagent system}, i.e., a system in $\mathbb R^d$ where the number of agents is so large that only a \emph{macroscopic}
description is available. As usual in this framework, in order to describe the behaviour of the system at a certain time $t$, we introduce a Borel positive measure $\mu_t$ on $\mathbb R^d$
whose meaning is the following: given a Borel set $A\subseteq \mathbb R^d$ the quantity
$\dfrac{\mu_t(A)}{\mu_t(\mathbb R^d)}$ represents the fraction of the total number of agents that are present in $A$ at the time $t$.
We will assume that the system is \emph{isolated}, thus the total number of agents remains constant in time. Hence, by normalizing the measure $\mu_t$,
we can always assume $\mu_t(\mathbb R^d)=1$, i.e., $\mu_t$ is a probability measure for all $t$.
\par\medskip\par
The macroscopic evolution of the system is thus given by a curve $t\mapsto\mu_t$ in the space of probability measures.
Due to the mass-preserving character of the evolution, we can assume that such an evolution is governed by the \emph{continuity equation}
\[\partial_t\mu_t+\mathrm{div}(v_t\mu_t)=0,\]
to be satisfied in a distributional sense, where $v_t$ is a suitable time-depending Borel vector field describing the macroscopic mass flux during the evolution.
\par\medskip\par
It can be easily proved, see e.g. \cite{CM-LSSC17}, that for a.e. $t$ and $\mu_t$-a.e. $x\in\mathbb R^d$ the vector field $v_t(x)$ can be constructed as a weigthed average of the velocities of the agents passing through 
the point $x$ at time $t$, where the weights are given by the fraction of the mass carried by each agent w.r.t.
the total amount of mass flowing through $x$ at time $t$.
In particular, possibly nonlocal nonholonomic constraints on the agents' motion will reflect into constraints for the possible choices of $v_t$.
\par\medskip\par
In this paper we consider a situation where each agent is constrained to follow the trajectories of a differential inclusion with a \emph{nonlocal} dependence
on the overall configuration of the agents. This fact models the possible \emph{nonlocal interaction} among the agents.
Examples of such interactions are quite commmon in the models of pedestrian dynamics, flocks of animals and social dynamics in general.
\par\medskip\par
Due to the potential applications, the literature on control of multi-agent systems is growing quite fast in the recent years. 
Among the most recent contributions, we mention \cite{CFPT}, where the authors investigate a controllability problem for a leader-follower model in a finite-dimensional setting
and their aim is to achieve an alignment consensus for a mass of indistinguishable agents when the action of an external policy maker 
is \emph{sparse}, i.e. concentrated on few individuals. In \cite{FPR} it is provided a mean-field formulation of the same model through Gamma-convergence techniques. 
\par
The relevance of such kind of results is enhanced when dealing with problems involving a considerable number of individuals, 
in order to circumvent the bounds coming from the curse of dimensionality: indeed, the mean-field limit can be used as a realistic approximation
when the number of agents is huge. Results in this direction are provided for example by \cite{FLOS} or the preprint paper \cite{CLOS}, 
where the authors study a Gamma-convergence result for an optimal control problem of a $N$-particles system subject to a nonlocal dynamics 
when $N\to +\infty$.
\par
Controllability conditions in the space of probability measures are also analyzed in the preprints \cite{DMR1}, \cite{DMR2}. 
In particular, the aim of the authors is to provide sufficient conditions in order to steer an initial configuration of agents into a 
desired final one, by acting through a control term on the vector field, under the constraint that the action can be implemented only in 
a certain fixed space region.
\par
Also the extension of classical \emph{viability theory} to multi-agent systems is attracting an increasing interest in the community.
Similarly to the finite-dimensional framework, a subset $\mathscr K$ of probability measures is said to be viable for a controlled dynamics if it is possible to 
keep the evolution confined inside $\mathscr K$ by acting with an admissible control when starting with a initial state in $\mathscr K$.
We refer to \cite{Aver} for first results in this direction based on a geometric approach (tangent cones to $\mathscr K$) and to the preprint \cite{CMQ} for a viscous-type approach to the problem.
\par
It is worth pointing out that a key feature of all these studies, and many others available in the literature, is the combined use of tools, concepts, and techniques from optimal transport theory, measure theory,
and from optimal control theory.
\par\medskip\par
In this paper, we deal with a time-optimal control problem. More precisely, given a target set of desired final configurations, we are interested in the minimum time needed 
to steer the agents to it starting from an initial distribution, and respecting the nonholonomic constraints. In particular, in our measure-theoretic setting, the target set is given in duality 
with the space $C^0_b(\mathbb R^d)$ of continuous and bounded functions as follows. Given a family of observables $\Phi\subseteq C^0_b(\mathbb R^d)$, the target set is defined as 
(see Definition \ref{def:gentar})
\begin{equation}\label{eq:gentar2}\tilde S_p^\Phi:=\left\{\mu\in\mathscr P_p(\mathbb R^d):\,\int_{\mathbb R^d}\varphi(x)\,d\mu(x)\le 0\textrm{ for all }\varphi\in\Phi\right\},\end{equation}
where $(\mathscr P_p(\mathbb R^d),W_p)$, with $p\ge1$, is the $p$-Wasserstein space of probability measures endowed with the metric $W_p$ (see Definition \ref{def:Wpspace}).
Section \ref{sec:gentarget} is entirely devoted to the analysis of topological properties of this class of generalized target sets.
In Section \ref{sec:admtraj}, we study the set of admissible trajectories $\mathcal A^p_I(\mu)$ defined on a time interval $I\subset\mathbb R$ and starting from a given initial 
datum $\mu\in\mathscr P_p(\mathbb R^d)$, i.e., those absolutely continuous curves in $\mathscr P_p(\mathbb R^d)$ whose driving velocity field satisfies the nonlocal nonholonomic 
constraint given by a set-valued map $F:\mathscr P_p(\mathbb R^d)\times\mathbb R^d\rightrightarrows \mathbb R^d$
\begin{align*}\partial_t\mu_t+\mathrm{div}(v_t\mu_t)=0,&&v_t(x)\in F(\mu_t,x)\textrm{ for a.e. }t\in I\textrm{ and $\mu_t$-a.e. }x\in\mathbb R^d.\end{align*}
This equation represents the controlled dynamics of the system.
The results obtained, expecially Theorem \ref{thm:filippov} providing Filippov-Gronwall type estimates, are then used in Section \ref{sec:mintime} where we study the 
main object of the paper, i.e., the \emph{minimum time function} defined as
\[\tilde T_p(\mu):=\inf\{T\ge 0:\, \textrm{ there exists }\boldsymbol\mu=\{\mu_t\}_{t\in[0,T]}\in\mathcal A^p_{[0,T]}(\mu)\textrm{ s.t. }\mu_T\in\tilde S_p^\Phi\}.\]
More precisely, we first prove the existence of optimal trajectories, lower semicontinuity of the minimum time function and a Dynamic Programming Principle as in the classical case.
\par\medskip\par
In the case without interactions, by passing to the limit in the Dynamic Programming Principle,
in \cite{CMNP} the authors proved that the minimum time function solves in a suitable viscosity sense an Hamilton-Jacobi-Bellman equation in the spaces of measures provided that it is \emph{continuous} (not just l.s.c.),
and further development on this theory have been recently done in \cite{MQ,JMQ}. We refer the reader to \cite{AGa,GT2019,GS} for an introduction to Hamilton-Jacobi
equations in Wasserstein spaces.
\par\medskip\par
The aim of this paper is to provide a sufficient condition for the continuity of the minimum time function in this framework, i.e., sufficient conditions granting Small Time Local Attainability (STLA) in the sense of \cite{KQ}.
Indeed, assuming STLA, in Proposition \ref{prop:STLAcontT} we get the continuity of the minimum time function thanks to the Filippov estimate proved in Theorem \ref{thm:filippov} and the Dynamic Programming Principle.
Sufficient conditions for STLA are finally provided in the main Theorem \ref{thm:suffSTLA}, assuming geometric properties of the generalized target set $\tilde S_p^\Phi$ together with a sort of gradient-descent behavior of the associated family of observables $\Phi$ when integrated along admissible trajectories (see Definition \ref{def:Sattain}). This condition represents a weakening of the well-known Petrov's condition in the classical framework.
\par\medskip\par
The paper is structured as follows: in Section \ref{sec:notation} we fix the notation and review some basic results about measure theory, 
optimal transport and set-valued analysis, in Section \ref{sec:admtraj} we prove some basic properties of the admissible trajectories in the space of measures,
in Section \ref{sec:gentarget} we discuss some geometric properties of the target sets. In Section \ref{sec:mintime} we state our main result concerning
the continuity of the minimum time function, and finally in Section \ref{sec:final} we compare our sufficient condition for STLA with the finite-dimensional one of \cite{KQ}.

\section{Preliminaries and notation}\label{sec:notation}

In this section we review some concepts from measure theory, optimal transport and set-valued analysis.
Our main references for this part are \cite{AGS}, \cite{AuF}, and \cite{Vil}.\par\medskip\par

We will use the following notation.\par\medskip\par
{\small\begin{longtable}{ll}
$B(x,r)$&the open ball of center $x\in X$ and radius $r$ of a normed space $X$,\\ &i.e., $B(x,r):=\{y\in X:\,\|y-x\|_X<r\}$;\\
$\overline K$&the closure of a subset $K$ of a topological space $X$;\\
$d_K(\cdot)$&the distance function from a subset $K$ of a metric space $(X,d)$,\\ &i.e. $d_K(x):=\inf\{d(x,y):\,y\in K\}$;\\
$C^0_b(X;Y)$&the set of continuous bounded function from a Banach space $X$ to $Y$, \\
&endowed with $\|f\|_{\infty}=\displaystyle\sup_{x\in X}\|f(x)\|_Y$ (if $Y=\mathbb R$, $Y$ will be omitted);\\
$C^0_c(X;Y)$&the set of compactly supported functions of $C^0_b(X;Y)$, \\
&with the topology induced by $C^0_b(X;Y)$;\\
$\Gamma_I$&the set of continuous curves from a real interval $I$ to $\mathbb R^d$;\\
$\Gamma_T$&the set of continuous curves from $[0,T]$ to $\mathbb R^d$;\\
$AC([0,T])$&the set of absolutely continuous curves from $[0,T]$ to $\mathbb R^d$;\\
$e_t$&the evaluation operator $e_t:\mathbb R^d\times\Gamma_I\to\mathbb R^d$\\ &defined by $e_t(x,\gamma)=\gamma(t)$ for all $t\in I$;\\
$\mathscr P(X)$&the set of Borel probability measures on a Banach space $X$,\\
&endowed with the weak$^*$ topology induced by $C^0_b(X)$;\\
$\mathscr M(\mathbb R^d;\mathbb R^d)$&the set of vector-valued Borel measures on $\mathbb R^d$ with values in $\mathbb R^d$,\\
&endowed with the weak$^*$ topology induced by $C^0_c(\mathbb R^d;\mathbb R^d)$;\\
$|\nu|$&the total variation of a measure $\nu\in \mathscr M(\mathbb R^d;\mathbb R^d)$;\\
$\ll$&the absolutely continuity relation between measures defined on the same\\
&$\sigma$-algebra;\\
$\mathrm{m}_p(\mu)$&the $p$-moment of a probability measure $\mu\in \mathscr P(X)$;\\
$r\sharp\mu$&the push-forward of the measure $\mu$ by the Borel map $r$;\\
$\mu\otimes\eta_x$&the product measure of $\mu\in\mathscr P(X)$ with the Borel family of measures\\ &$\{\eta_x\}_{x\in X}$;\\
$\mathrm{pr}_i$&the $i$-th projection map $\mathrm{pr}_i(x_1,\dots,x_N)=x_i$;\\
$\Pi(\mu,\nu)$&the set of admissible transport plans from $\mu$ to $\nu$;\\
$\Pi_o(\mu,\nu)$&the set of optimal transport plans from $\mu$ to $\nu$;\\
$W_p(\mu,\nu)$&the $p$-Wasserstein distance between $\mu$ and $\nu$;\\
$\mathscr P_p(X)$&the subset of the elements $\mathscr P(X)$ with finite $p$-moment, \\
&endowed with the $p$-Wasserstein distance;\\
$\mathscr L^d$&the Lebesgue measure on $\mathbb R^d$;\\
$\dfrac{\nu}{\mu}$&the Radon-Nikodym derivative of the measure $\nu$ w.r.t. the measure $\mu$;\\
$\mathrm{Lip}(f)$&the Lipschitz constant of a function $f$.\\
$\mathcal A^p_I(\mu)$&the set of admissible trajectories defined in \eqref{eq:admset}.\\
$\Upsilon_F(\mu,\boldsymbol\theta)$&the set defined in Definition \ref{def:bas}.\\
$K_F$&the quantity defined in Hypothesis \ref{HP} by $\displaystyle K_F:=\max_{v\in F(\delta_0,0)}\{|v|\}$.
\end{longtable}}

In this section we give some preliminaries and fix the notation. Our main reference for this part is \cite{AGS}.

\begin{definition}[Space of probability measures]
Given Banach spaces $X,Y$, we denote by $\mathscr P(X)$ the set of Borel probability measures on $X$ endowed with the weak$^*$ topology
induced by the duality with the Banach space $C^0_b(X)$ of the real-valued continuous bounded functions on $X$ with the uniform convergence norm.
For any $p\ge 1$, the $p$-moment of $\mu\in\mathscr P(X)$ is defined by $\displaystyle\mathrm{m}_p(\mu)=\int_{X}\|x\|_X^p\,d\mu(x)$, and
we set $\mathscr P_p(X)=\{\mu\in\mathscr P(X):\, \mathrm{m}_p(\mu)<+\infty\}$.
For any Borel map $r:X\to Y$ and $\mu\in\mathscr P(X)$, we define the \emph{push forward measure} $r\sharp\mu\in\mathscr P(Y)$
by setting $r\sharp\mu(B)=\mu(r^{-1}(B))$ for any Borel set $B$ of $Y$.
\end{definition}

\begin{definition}[Total variation]
Let $X,Y$ be Banach spaces, and denote by $\mathscr M(X;Y)$ the set of $Y$-valued Borel measures defined on $X$.
The total variation measure of $\nu\in\mathscr M(X;Y)$ is defined for every Borel set $B\subseteq X$ as
\[
|\nu|(B):=\sup_{\{B_i\}_{i\in\mathbb N}}\left\{\sum \|\nu(B_i)\|_Y\right\}
\]
where the sup ranges on the set of countable collections $\{B_i\}_{i\in\mathbb N}$ of pairwise disjoint Borel sets 
such that $\displaystyle\bigcup_{i\in\mathbb N}B_i=B$.
\end{definition}

For the following result see \cite[Theorem 5.3.1]{AGS}.

\begin{theorem}[Disintegration]
Given a measure $\mu\in\mathscr P(\mathbb X)$ and a Borel map $r:\mathbb X\to X$, there exists a family of probability
measures $\{\mu_x\}_{x\in X}\subseteq \mathscr P(\mathbb X)$,
uniquely defined for $r\sharp \mu$-a.e. $x\in X$, such that $\mu_x(\mathbb X\setminus r^{-1}(x))=0$ for $r\sharp \mu$-a.e. $x\in X$, and
for any Borel map $\varphi:\mathbb X\to [0,+\infty]$ we have
\[\int_{\mathbb X}\varphi(z)\,d\mu(z)=\int_X \left[\int_{r^{-1}(x)}\varphi(z)\,d\mu_x(z)\right]d(r\sharp \mu)(x).\]
We will write $\mu=(r\sharp \mu)\otimes \mu_x$.
If $\mathbb X=X\times Y$ and $r^{-1}(x)\subseteq\{x\}\times Y$ for all $x\in X$, we can identify each measure $\mu_x\in\mathscr P(X\times Y)$ with a measure on $Y$.
\end{theorem}

\begin{definition}[Transport plans and Wasserstein distance]\label{def:Wpspace}
Let $X$ be a complete separable Banach space, $\mu_1,\mu_2\in\mathscr P(X)$. We define the set of \emph{admissible transport plans} between $\mu_1$ and $\mu_2$ by setting
\[\Pi(\mu_1,\mu_2)=\{\boldsymbol{\pi}\in\mathscr P(X\times X): \mathrm{pr}_i\sharp\boldsymbol\pi=\mu_i,\,i=1,2\},\]
where for $i=1,2$, we defined $\mathrm{pr}_i:\mathbb R^d\times\mathbb R^d\to\mathbb R^d$ by $\mathrm{pr}_i(x_1,x_2)=x_i$.
The \emph{inverse} $\boldsymbol\pi^{-1}$ of a transport plan $\boldsymbol\pi\in \Pi(\mu,\nu)$ is defined by $\boldsymbol\pi^{-1}=i\sharp \boldsymbol\pi\in \Pi(\nu,\mu)$,
where $i(x,y)=(y,x)$ for all $x,y\in X$.
The \emph{$p$-Wasserstein distance} between $\mu_1$ and $\mu_2$ is
\[W_p^p(\mu_1,\mu_2)=\inf_{\boldsymbol\pi\in\Pi(\mu_1,\mu_2)}\int_{X\times X}|x_1-x_2|^p\,d\boldsymbol\pi(x_1,x_2).\]
If $\mu_1,\mu_2\in\mathscr P_p(X)$ then the above infimum is actually a minimum, and we define
\[\Pi^p_o(\mu_1,\mu_2)=\left\{\boldsymbol\pi\in\Pi(\mu_1,\mu_2): W_p^p(\mu_1,\mu_2)= \int_{X\times X}|x_1-x_2|^p\,d\boldsymbol\pi(x_1,x_2)\right\}.\]
The space $\mathscr P_p(X)$ endowed with the $W_p$-Wasserstein distance is a complete separable metric space, moreover
for all $\mu\in \mathscr P_p(X)$ there exists a sequence $\{\mu^N\}_{N\in\mathbb N}\subseteq\mathrm{co}\{\delta_x:\,x\in\mathrm{supp}\,\mu\}$ such that $W_p(\mu^N,\mu)\to 0$ as $N\to +\infty$.
\end{definition}

\begin{remark}
Recalling formula (5.2.12) in \cite{AGS}, we have 
\[W_p(\delta_0,\mu)=m_{p}^{1/p}(\mu)=\left(\int_{\mathbb R^d}|x|^p\,d\mu(x)\right)^{1/p}\]
for all $\mu\in \mathscr P_p(\mathbb R^d)$. In particular, if $t\mapsto \mu_t$ is $W_p$-continuous, then $t\mapsto \mathrm{m}_p^{1/p}(\mu_t)$ is continuous.
\end{remark}

\begin{definition}[Set-valued maps]
Let $X,Y$ be sets. A \emph{set-valued} map $F$ from $X$ to $Y$ is a map associating to each $x\in X$ a (possible empty) subset $F(x)$ of $Y$. We will write $F:X\rightrightarrows Y$ to denote a set-valued map from $X$ to $Y$.
The \emph{graph} of a set-valued map $F$ is 
\[\mathrm{graph}\,F:=\{(x,y)\in X\times Y:\,y\in F(x)\}\subseteq X\times Y,\]
while the \emph{domain} of $F$ is $\mathrm{dom}\,F:=\{x\in X:\, F(x)\ne\emptyset\}\subseteq X$.
A \emph{selection} of $F$ is a map $f:\mathrm{dom}\,F\to Y$ such that $f(x)\in F(x)$ for all $x\in\mathrm{dom}\,F$.
When $X,Y$ are topological spaces, we say that
\begin{itemize}
\item $F$ has \emph{closed images} if $F(x)$ is closed in $Y$ for every $x\in X$, 
\item $F$ has \emph{closed graph} if $\mathrm{graph}\,F$ is closed in $X\times Y$,
\item $F$ is \emph{compact valued} (or that it has \emph{compact images}) if $F(x)$ is compact for every $x\in X$,
\item $F$ is \emph{upper semicontinuous} at $x\in X$ if for every open set $V\subseteq Y$ such that $V\supseteq F(x)$
there exists an open neighborhood $U\subseteq X$ of $x$ such that $F(z)\subseteq V$ for all $z\in U$. 
\item $F$ is \emph{lower semicontinuous} at $x\in X$ if for every open set $V\subseteq Y$ such that $V\cap F(x)\ne \emptyset$ there exists an open neighborhood $U\subseteq X$ of $x$ such that $F(z)\cap V\ne\emptyset$ for all $z\in U$. 
\item $F$ is continuous at $x\in X$ if it is both lower and upper semicontinuous at $x$. 
\item $F$ will be called continuous (resp. lower semicontinuous, upper semicontinuous) if it is continuous (resp. lower semicontinuous, upper semicontinuous) at every $x\in X$.
\end{itemize}
When $Y$ is a vector space, $F$ is convex valued (or it has \emph{convex images}) if $F(x)$ is convex for every $x\in X$.
When $X,Y$ are measurable spaces, we say that $F$ is measurable if $\mathrm{graph}\,F$ is measurable in $X\times Y$ endowed with the product of $\sigma$-algebrae on $X$ and $Y$.
When $(X,d)$ is a metric space and $Y$ is a normed space, given $L>0$ we say that $F$ is Lipschitz continuous with constant $L$ if for all $x_1,x_2\in X$
\[F(x_2)\subseteq F(x_1)+L\cdot d(x_1,x_2)\overline{B_Y(0,1)},\]
where the sum and the product of sets are in the Minkowski sense: $A+B=\{a+b:\,a\in A,\,b\in B\}$ and $\lambda A=\{\lambda a:\,a\in A\}$ for every $A,B\subseteq Y$, $\lambda\in\mathbb R$.
\end{definition}
 
\section{Admissible trajectories}\label{sec:admtraj}
Given a collection $\boldsymbol\mu=\{\mu_h\}_{h\in I}\subseteq\mathscr P(X)$ of Borel measures on the measure space $X$ indexed by a parameter $h\in I$, by a slight abuse of notation we will denote by $\boldsymbol\mu$ both the set $\boldsymbol\mu=\{\mu_h\}_{h\in I}\subseteq\mathscr P(X)$ and the function $h\mapsto \mu_h$. In each occurrence, the context will clarify what we are referring to.

\begin{definition}[Admissible trajectories]
Let $I=[a,b]$ be a compact real interval, $\boldsymbol\mu=\{\mu_t\}_{t\in I}\subseteq \mathscr P_p(\mathbb R^d)$, $\boldsymbol\nu=\{\nu_t\}_{t\in I}\subseteq \mathscr M(\mathbb R^d;\mathbb R^d)$,
$F:\mathscr P_p(\mathbb R^d)\times\mathbb R^d \rightrightarrows \mathbb R^d$ be a set-valued map.\par
We say that $\boldsymbol\mu$ is an \emph{admissible trajectory driven by} $\boldsymbol\nu$ defined on $I$ with underlying dynamics $F$ if
\begin{itemize}
\item $|\nu_t|\ll \mu_t$ for a.e. $t\in I$;
\item $v_t(x):=\dfrac{\nu_t}{\mu_t}(x)\in F(\mu_t,x)$ for a.e. $t\in I$ and $\mu_t$-a.e. $x\in \mathbb R^d$;
\item $\partial_t\mu_t+\mathrm{div}\,\nu_t=0$ in the sense of distributions on $[0,T]\times\mathbb R^d$, equivalently 
\[\dfrac{d}{dt}\int_{\mathbb R^d}\varphi(x)\,d\mu_t(x)=\int_{\mathbb R^d}\langle\nabla\varphi(x),v_t(x)\rangle\,d\mu_t(x),\]
for a.e. $t\in [0,T]$ and all $\varphi\in C^1_c(\mathbb R^d)$.
\end{itemize}
Define $\mathcal A^p_I:\mathscr P_p(\mathbb R^d)\rightrightarrows C^0([0,T];\mathscr P_p(\mathbb R^d))$ by
\begin{equation}\label{eq:admset}
\mathcal A^p_I(\mu):=\left\{\boldsymbol\mu=\{\mu_t\}_{t\in I}:\, \begin{array}{l}\boldsymbol\mu\textrm{ is an admissible trajectory}\\ \textrm{with }\mu_a=\mu\end{array}\right\}.
\end{equation}
When $I$ and $p$ are clear by the context, we will omit them.
\end{definition}

\begin{definition}[Concatenation, restriction, extension]
Let $I_i=[a_i,b_i]$, $\boldsymbol\mu^{(i)}=\{\mu^{(i)}_t\}_{t\in I_i}\in\mathcal A^p_{I_i}(\mu^{(i)}_{a_i})$, $i=1,2$, be satisfying $\mu^{(1)}_{b_1}=\mu^{(2)}_{a_2}$. We define $I_3=[a_1,b_1+b_2-a_2]$ and
$\boldsymbol\mu^{(3)}=\{\mu^{(3)}_t\}_{t\in I_3}$ by setting $\mu^{(3)}_t=\mu^{(1)}_t$ for $t\in I_1$ and 
$\mu^{(3)}_{t}=\mu^{(2)}_{t+a_2-b_1}$ for $t\in I_3\setminus I_1$. The curve $\boldsymbol\mu^{(3)}$
will be called the \emph{concatenation} of $\boldsymbol\mu^{(1)}$ and $\boldsymbol\mu^{(2)}$ and will be denoted by
$\boldsymbol\mu^{(3)}=\boldsymbol\mu^{(1)}\odot\boldsymbol\mu^{(2)}$. By \cite[Lemma 4.4]{DNS}, we have 
$\boldsymbol\mu^{(3)}\in\mathcal A_{I_3}^p(\mu^{(1)}_{a_1})$.
\par\medskip\par
Let $I=[a,b]$, $J=[a',b']$ with $J\subseteq I$, and $\boldsymbol\mu=\{\mu_t\}_{t\in I}\in\mathcal A^p_{I}(\mu_a)$.
The \emph{restriction} $\boldsymbol\mu_{|J}=\{\hat\mu_t\}_{t\in J}$ of $\boldsymbol\mu$ to $J$ is defined by
taking $\hat\mu_t=\mu_t$ for all $t\in J$ and we have $\boldsymbol\mu_{|J}\in\mathcal A^p_{J}(\mu_{a'})$.
\par\medskip\par
Let $\boldsymbol\mu^{(i)}\in\mathcal A^p_{I_i}(\mu^{(i)})$, $i=1,2$. We say that $\boldsymbol\mu^{(2)}$
is an \emph{extension} of $\boldsymbol\mu^{(1)}$ if $I_2\supseteq I_1$ and $\boldsymbol\mu^{(2)}_{|I_1}=\boldsymbol\mu^{(1)}$.
\end{definition}

Throughout the paper, we will assume the following
\begin{hypothesis}\label{HP}
The set-valued map $F:\mathscr P_p(\mathbb R^d)\times\mathbb R^d\to \mathbb R^d$ has nonempty, convex and compact images, moreover 
it is Lipschitz continuous with constant $L>0$ with respect to the metric 
\[d_{\mathscr P_p(\mathbb R^d)\times\mathbb R^d}((\mu_1,x_1),(\mu_2,x_2))=W_p(\mu_1,\mu_2)+|x_1-x_2|\] 
on $\mathscr P_p(\mathbb R^d)\times\mathbb R^d$. We set $\displaystyle K_F:=\max_{v\in F(\delta_0,0)}\{|v|\}$.
\end{hypothesis}

\begin{definition}[Definition of $\Upsilon_F$]\label{def:bas}
Assume Hypothesis \ref{HP} for $F$. Let $\boldsymbol\theta=\{\theta_t\}_{t\in [0,T]}$ be a $W_p$-continuous curve in $\mathscr P_p(\mathbb R^d)$, $\mu\in\mathscr P_p(\mathbb R^d)$.
Denote by $\Upsilon_F(\mu,\boldsymbol\theta)$ the set of $\boldsymbol\mu=\{\mu_t\}_{t\in [0,T]}\subseteq\mathscr P_p(\mathbb R^d)$ satisfying
the following property: there exists $\boldsymbol\eta\in\mathscr P(\mathbb R^d\times\Gamma_T)$ such that
\begin{itemize}
\item $\mu_t=e_t\sharp\boldsymbol\eta$ for all $t\in [0,T]$, $\mu_0=\mu$;
\item for $\boldsymbol\eta$-a.e. $(x,\gamma)\in\mathbb R^d\times\Gamma_T$ and a.e. $t\in[0,T]$ it holds $\gamma\in AC([0,T])$, $\gamma(0)=x$, $\dot\gamma(t)\in F(\theta_t,\gamma(t))$ .
\end{itemize}
We set $\displaystyle M_{\boldsymbol\theta}:=\sup_{\tau\in[0,T]}\mathrm{m}_p^{1/p}(\theta_\tau)$ and
\[\Xi(\mu,\boldsymbol\theta):=\{\boldsymbol\eta\in\mathscr P(\mathbb R^d\times\Gamma_T):\, \{e_t\sharp\boldsymbol\eta\}_{t\in[0,T]}\in \Upsilon_F(\mu,\boldsymbol\theta)\}.\]
On the set $X:=\mathscr P_p(\mathbb R^d)\times C^0([0,T];\mathscr P_p(\mathbb R^d))$ we define the metric
\[d_{X}\left(\left(\mu^{(1)},\boldsymbol{\theta^{(1)}}\right),\left(\mu^{(2)},\boldsymbol{\theta^{(2)}}\right)\right)
=W_p\left(\mu^{(1)},\mu^{(2)}\right)+\sup_{t\in[0,T]}W_p\left(\theta^{(1)}_t,\theta^{(2)}_t\right),\]
where $\boldsymbol{\theta^{(i)}}=\{\theta^{(i)}_t\}_{t\in[0,T]}$, $i=1,2$.\par
Finally, we define the set-valued map $S^{\boldsymbol\theta}:\mathbb R^d\rightrightarrows \Gamma_T$ by setting for all $y\in\mathbb R^d$
\begin{align*}
S^{\boldsymbol\theta}(y):=&\{\xi\in AC([0,T]):\, \dot\xi\in F(\theta_t,\xi(t))\textrm{ for a.e. }t\in[0,T],\,\xi(0)=y\}.
\end{align*}
Since $t\mapsto F(\theta_t,x)$ is continuous for all $x\in\mathbb R^d$ and $x\mapsto F(\theta_t,x)$ is Lipschitz continuous by Hypothesis \ref{HP}, with constant $L$ 
for all $t\in[0,T]$, the set-valued map $S^{\boldsymbol\theta}(\cdot)$ is Lipschitz continuous by  \cite[Corollary 10.4.2]{AuF}.
\end{definition}

\begin{lemma}[Estimates on the moments]\label{lemma:moment}
Assume Hypothesis \ref{HP} for $F$.
Let $\boldsymbol\theta=\{\theta_t\}_{t\in [0,T]}$ be a $W_p$-continuous curve in $\mathscr P_p(\mathbb R^d)$, $\mu\in\mathscr P_p(\mathbb R^d)$. Let $\Xi(\mu,\boldsymbol\theta)$, $\Upsilon_F(\mu,\boldsymbol\theta)$ and $M_{\boldsymbol\theta}$ be as in Definition \ref{def:bas}, and let $\boldsymbol \eta\in \Xi(\mu,\boldsymbol\theta)$ and $\boldsymbol\mu=\{\mu_t\}_{t\in [0,T]}\in \Upsilon_F(\mu,\boldsymbol\theta)$ such that $\mu_t=e_t\sharp\boldsymbol\eta$ for all $t\in [0,T]$.
Then for all $t,s\in[0,T]$
\begin{align*}
\mathrm{m}_p^{1/p}(\mu_t)\le&e^{LT}\left(\mathrm{m}_p^{1/p}(\mu)+K_FT+LTM_{\boldsymbol\theta}\right),\\
W_p(\mu_t,\mu_s)\le&\left(K_F+LM_{\boldsymbol\theta}+Le^{LT}\left(\mathrm{m}_p^{1/p}(\mu)+K_FT+LTM_{\boldsymbol\theta}\right)\right)\cdot|t-s|,\\ 
\int_{\mathbb R^d\times\Gamma_T}\|\dot\gamma\|^p_{L^{\infty}([0,1])}&\,d\boldsymbol\eta(x,\gamma)\le \left[K_F+L\left(e^{LT}\left(\mathrm{m}_p^{1/p}(\mu)+K_FT+LTM_{\boldsymbol\theta}\right)+M_{\boldsymbol\theta}\right)\right]^p.
\end{align*}
\end{lemma}
\begin{proof}
Set $\boldsymbol\mu=\{\mu_t\}_{t\in[0,T]}\in\Upsilon_F(\mu,\boldsymbol\theta)$.
For $\boldsymbol\eta$-a.e. $(x,\gamma)\in\mathbb R^d\times\Gamma_T$ and a.e. $t\in[0,T]$ we have
\[\dot \gamma(t)\in F(\theta_t,\gamma(t))\subseteq F(\delta_0,0)+L(|\gamma(t)|+\mathrm{m}_p^{1/p}(\theta_t))\overline{B(0,1)}\] 
Thus for all $s,t\in [0,T]$ and a.e. $\tau\in[0,T]$
\begin{align*}
|\dot\gamma(\tau)|\le &K_F+L(|\gamma(\tau)|+\mathrm{m}_p^{1/p}(\theta_\tau))\le K_F+LM_{\boldsymbol\theta}+L|\gamma(\tau)|,\\
|\gamma(t)|-|\gamma(0)|\le&\int_0^t|\dot\gamma(\tau)|\,d\tau\le (K_F+LM_{\boldsymbol\theta})T+L\int_0^t|\gamma(\tau)|\,d\tau,\\
|\gamma(t)-\gamma(s)|\le&\left|\int_s^t|\dot\gamma(\tau)|\,d\tau\right|\le (K_F+LM_{\boldsymbol\theta})|t-s|+L\left|\int_s^t|\gamma(\tau)|\,d\tau\right|.
\end{align*}
By Gr\"onwall lemma, this implies for all $0\le s\le t\in [0,T]$ 
\begin{align}\nonumber
|\gamma(t)|\le&e^{Lt}\left(|\gamma(0)|+(K_F+LM_{\boldsymbol\theta})T\right),\\ \label{eq:Lip}
\|\dot\gamma\|_{L^{\infty}([0,1])}\le &K_F+L\left(M_{\boldsymbol\theta}+e^{LT}\left(|\gamma(0)|+(K_F+LM_{\boldsymbol\theta})T\right)\right),\\ \nonumber
|\gamma(t)-\gamma(s)|\le& \left(K_F+LM_{\boldsymbol\theta}+Le^{LT}\left(|\gamma(0)|+(K_F+LM_{\boldsymbol\theta})T\right)\right)\cdot |t-s|,
\end{align}
recalling that $\dot\gamma(s)\in F(\mu_s,\gamma(s))$ for a.e. $s$.\par
We conclude by taking the $L^p_{\boldsymbol\eta}$ norm of the above inequalities and using the triangular inequality.
\end{proof}

\begin{proposition}[Upper semicontinuity of the solution map]\label{prop:uscsol}
Set $X:=\mathscr P_p(\mathbb R^d)\times C^0([0,T];\mathscr P_p(\mathbb R^d))$.\par
The set-valued map $\Upsilon_F:X\rightrightarrows C^0([0,T];\mathscr P_p(\mathbb R^d))$, defined in Definition \ref{def:bas}, is upper semicontinuous with compact nonempty images.
\end{proposition}
\begin{proof}
We prove first that $\Upsilon_F(\mu,\boldsymbol\theta)\ne\emptyset$ for all $(\mu,\boldsymbol\theta)\in X$. 
Consider now the set-valued map $S^{\boldsymbol\theta}(\cdot)$ defined as in Definition \ref{def:bas}.
Since it is Lipschitz continuous, it has a Borel selection. Thus let $h_0:\mathbb R^d\to AC([0,T])$ be a Borel map such that $h_0(x)\in S^{\boldsymbol\theta}(x)$ for every $x\in\mathbb R^d$.
Define $\boldsymbol\eta=\mu\otimes\delta_{h_0(x)}$, $\mu_t=e_t\sharp\boldsymbol\eta$, $\boldsymbol\mu=\{\mu_t\}_{t\in[0,T]}$. Then, by construction, we have $\boldsymbol\mu\in\Upsilon_F(\mu,\boldsymbol\theta)$. 
\par\medskip\par
Let now $\left\{\left(\mu^{(n)},\boldsymbol\theta^{(n)}\right)\right\}_{n\in\mathbb N}\subseteq X$ be a sequence $d_X$-converging to $(\mu,\boldsymbol\theta)\in X$,
and $\{\boldsymbol\mu^{(n)}\}_{n\in\mathbb N}\subseteq C^0([0,T];\mathscr P_p(\mathbb R^d))$, $\{\boldsymbol\eta^{(n)}\}_{n\in\mathbb N}\subseteq \mathscr P(\mathbb R^d\times\Gamma_T)$ 
be such that 
\begin{itemize}
\item $\boldsymbol\theta^{(n)}=\{\theta^{(n)}_t\}_{t\in[0,T]}$, $\boldsymbol\theta=\{\theta_t\}_{t\in[0,T]}$;
\item $\boldsymbol\mu^{(n)}\in \Upsilon_F(\mu^{(n)},\boldsymbol\theta^{(n)})$ and $\boldsymbol\eta^{(n)}\in\Xi(\mu^{(n)},\boldsymbol\theta^{(n)})$ for all $n\in\mathbb N$;
\item $\boldsymbol\mu^{(n)}=\{\mu^{(n)}_t\}_{t\in[0,T]}$ with $\mu^{(n)}_t=e_t\sharp\boldsymbol\eta^{(n)}$ for all $t\in[0,T]$ and $n\in\mathbb N$,
\end{itemize}
where $\Xi(\cdot,\cdot)$ is defined as in Definition \ref{def:bas}.
We prove that the sequence $\{\boldsymbol\mu^{(n)}\}_{n\in\mathbb N}$ has always cluster points, and all the cluster points are contained in $\Upsilon_F(\mu,\boldsymbol{\theta})$.
This will imply in particular that $\Upsilon_F(\cdot)$ has compact images (by taking constant sequences $(\mu^{(n)},\boldsymbol\theta^{(n)})\equiv(\mu,\boldsymbol\theta)$).
\par\medskip\par
For $n$ sufficiently large, we have $\mathrm{m}_p^{1/p}(\mu^{(n)})\le \mathrm{m}_p^{1/p}(\mu)+1$ and  $M_{\boldsymbol\theta^{(n)}}\le M_{\boldsymbol{\theta}}+1$,
recalling the definition of the convergence in $X$ and the definition for $M_{\boldsymbol\theta}$ given in Definition \ref{def:bas}.
Thus, by applying the estimates of Lemma \ref{lemma:moment}, we have
\begin{align*}
&\mathrm{m}_p^{1/p}(\mu^{(n)}_t)\le e^{LT}\left(m_p^{1/p}(\mu)+1+K_FT+LTM_{\boldsymbol\theta}+LT\right),\\
&W_p(\mu^{(n)}_t,\mu^{(n)}_s)\le\\ 
&\le\left(K_F+LM_{\boldsymbol\theta}+L+Le^{LT}\left(\mathrm{m}_p^{1/p}(\mu)+1+K_FT+LTM_{\boldsymbol\theta}+LT\right)\right)\cdot|t-s|,\\ 
&\int_{\mathbb R^d\times\Gamma_T}\|\dot\gamma\|^p_{L^{\infty}([0,1])}\,d\boldsymbol\eta^{(n)}(x,\gamma)\le\\ 
&\le\left[K_F+L\left(e^{LT}\left(\mathrm{m}_p^{1/p}(\mu)+1+K_F T+LTM_{\boldsymbol\theta}+LT\right)+M_{\boldsymbol\theta}+1\right)\right]^p.
\end{align*}
In particular
\begin{itemize}
\item $\{\boldsymbol\mu^{(n)}\}_{n\in\mathbb N}$ is equicontinuous;
\item for all $t\in [0,T]$, we have that $\{\mu^{(n)}_t\}_{n\in\mathbb N}$ is relatively compact in $\mathscr P_p(\mathbb R^d)$, since it has $p$-moment uniformly bounded.
\end{itemize}
Thus $\{\boldsymbol\mu^{(n)}\}_{n\in\mathbb N}$ is relatively compact in $C^0([0,T];\mathscr P_p(\mathbb R^d))$ by Ascoli-Arzel\`a theorem.
Up to passing to a subsequence, we may assume that there exists $\boldsymbol\mu=\{\mu_t\}_{t\in[0,T]}\in C^0([0,T];\mathscr P_p(\mathbb R^d))$ such that 
\[\lim_{n\to +\infty}\sup_{t\in [0,T]}W_p(\mu_t,\mu^{(n)}_t)=0.\]
We notice also that the functional $\Psi:\mathbb R^d\times\Gamma_T\to\mathbb R\cup\{+\infty\}$
\[\Psi(x,\gamma):=\begin{cases}(|x|+|\gamma(0)|+\|\dot\gamma\|_{L^\infty})^p,\textrm{ if }\gamma\in\mathrm{Lip}([0,T];\mathbb R^d),\\ \\ +\infty,\textrm{ otherwise},\end{cases}\]
has compact sublevels in $\mathbb R^d\times\Gamma_T$ by Ascoli-Arzel\`a theorem.
Since
\begin{multline*}
\sup_{n\in\mathbb N}\int_{\mathbb R^d\times\Gamma_T}\Psi(x,\gamma)\,d\boldsymbol\eta^{(n)}(x,\gamma)\le\\
\le 2^{p-1}\sup_{n\in\mathbb N}\left[2\mathrm{m}_p(\mu^{(n)})+\int_{\mathbb R^d\times\Gamma_T}\|\dot\gamma\|^p_{L^{\infty}([0,1])}\,d\boldsymbol\eta^{(n)}(x,\gamma)\right]<+\infty,
\end{multline*}
we have that $\{\boldsymbol\eta^{(n)}\}_{n\in\mathbb N}$ is tight in $\mathscr P(\mathbb R^d\times\Gamma_T)$. Thus, up to passing to a subsequence,
we may assume also that there exists $\boldsymbol\eta\in\mathscr P(\mathbb R^d\times\Gamma_T)$ such that $\boldsymbol\eta^{(n)}\rightharpoonup \boldsymbol\eta$ narrowly.
By the continuity of $e_t:\mathbb R^d\times\Gamma_T\to\mathbb R^d$, we have $\mu_t=e_t\sharp\boldsymbol\eta$.
By \cite[Proposition 5.1.8]{AGS}, for $\boldsymbol\eta$-a.e. $(x,\gamma)\in \mathbb R^d\times\Gamma_T$ there exists a sequence $\{(x_n,\gamma_n)\}_{n\in\mathbb N}$
such that $x_n=\gamma_n(0)$, $\gamma_n\in AC([0,T])$, $\dot\gamma_n(t)\in F(\theta_t^{(n)},\gamma_n(t))$ for a.e. $t\in [0,T]$ and for all $n\in\mathbb N$
with $x_n\to x$, $\|\gamma_n-\gamma\|_{\infty}\to 0$ as $n\to+\infty$.
\par\medskip\par
By \eqref{eq:Lip} and recalling that $M_{\boldsymbol\theta^{(n)}}\le M_{\boldsymbol\theta}+1$ and $|x_n|\le |x|+1$ for $n$ sufficiently large,
we have $n\in\mathbb N$,
\[\|\dot\gamma_n\|_{L^{\infty}([0,1])}\le K_F+L\left(M_{\boldsymbol\theta}+1+e^{LT}\left(|x|+1+(K_F+LM_{\boldsymbol\theta}+L)T\right)\right).\]
In particular, by Ascoli-Arzel\`a Theorem, we have that $\gamma$ is Lipschitz continuous.
For a.e. $t,\tau\in[0,T]$ we have also
\begin{align*}
F&(\theta^{(n)}_\tau,\gamma_n(\tau))\subseteq F(\theta_t,\gamma(t))+\\ 
&+L\left(W_p(\theta^{(n)}_\tau,\theta_\tau)+W_p(\theta_{\tau},\theta_t)+|\gamma_n(\tau)-\gamma(\tau)|+|\gamma(t)-\gamma(\tau)|\right)\overline{B(0,1)}\\
\subseteq&F(\theta_t,\gamma(t))+L\left(\sup_{\tau\in[0,T]}W_p(\theta^{(n)}_\tau,\theta_\tau)+W_p(\theta_\tau,\theta_t)+\|\gamma_n-\gamma\|_{\infty}+\mathrm{Lip}(\gamma)\cdot |t-\tau|\right)\overline{B(0,1)}
\end{align*}
For every $\varepsilon>0$ there is $n_{\varepsilon}\in\mathbb N$ such that if $n>n_\varepsilon$ we have for a.e. $t,\tau\in[0,T]$
\[\dot\gamma_n(\tau)\in F(\theta_t,\gamma(t))+L(\varepsilon+W_p(\theta_\tau,\theta_t)+\mathrm{Lip}(\gamma)|t-\tau|)\overline{B(0,1)}.\]
In particular, let $t\in [0,T]$ be a differentiability point of $\gamma_n$. We have for all $z\in\mathbb R^d$, $s\in[0,T]$, $s\ne t$, and $n>n_\varepsilon$
\begin{align*}
\langle \dfrac{\gamma_n(s)-\gamma_n(t)}{s-t},z\rangle=&\dfrac{1}{s-t}\int_t^s \langle z,\dot\gamma_n(\tau)\rangle\,d\tau\\ 
\le&\sup_{v\in F(\theta_t,\gamma(t))}\langle z,v\rangle+L\varepsilon|z|+L|z|\mathrm{Lip}(\gamma)\dfrac{1}{s-t}\int_t^s|t-\tau|\,d\tau+L|z|\frac{1}{s-t}\int_s^tW_p(\theta_\tau,\theta_t)\,d\tau
\end{align*}
By letting $n\to+\infty$ and $s\to t$ we conclude that $\dot\gamma(t)\in F(\theta_t,\gamma(t))$ since $F(\theta_t,\gamma(t))$ is closed and convex.
Hence $\boldsymbol\mu\in \Upsilon_F(\mu,\boldsymbol\theta)$, which completes the proof.
\end{proof}

\begin{proposition}[Superposition Principle]\label{prop:SP}
Assume Hypothesis \ref{HP} for $F$, and let $T>0$.
Then $\boldsymbol\mu=\{\mu_t\}_{t\in [0,T]}$ is an admissible trajectory if and only if $\boldsymbol\mu\in \Upsilon_F(\mu_0,\boldsymbol\mu)$, with $\Upsilon_F(\cdot,\cdot)$ defined in Definition \ref{def:bas}.
\end{proposition}
\begin{proof}
\begin{enumerate}
\item[]
\item \emph{Sufficience}. Assume that $\boldsymbol\mu=\{\mu_t\}_{t\in[0,T]}\in\Upsilon_F(\mu_0,\boldsymbol\mu)$. Let $\Xi(\cdot,\cdot)$ be as in Definition \ref{def:bas}.
Then there exists $\boldsymbol\eta\in \Xi(\mu_0,\boldsymbol\mu)$ such that $\mu_t=e_t\sharp\boldsymbol\eta$.
Set
\[\mathscr N:=\left\{(t,x,\gamma)\in [0,T]\times\mathbb R^d\times\Gamma_T:\, \nexists\dot\gamma(t)\textrm{ or }\dot\gamma(t)\notin F(\mu_t,\gamma(t))\textrm{ or }\gamma(0)\ne x\right\}.\]
Since $\mathscr L^1\otimes\boldsymbol\eta\left(\mathscr N\right)=0$, 
for $\boldsymbol\eta$-a.e. $(x,\gamma)\in\mathbb R^d\times\Gamma_T$ and a.e. $t\in [0,T]$ we have that $\dot\gamma(t)$ exists and belongs to $F(\mu_t,\gamma(t))$, and $\gamma(0)=x$.
Given $\varphi\in C^{1}_c(\mathbb R^d)$, we have
\[\left|\int_{\mathbb R^d}\varphi(x)\,d\mu_t(x)-\int_{\mathbb R^d}\varphi(x)\,d\mu_s(x)\right|\le \|\nabla\varphi\|_{\infty}\int_{\mathbb R^d\times\Gamma_T}|\gamma(t)-\gamma(s)|\,d\boldsymbol\eta(x,\gamma).\] 
According to  \eqref{eq:Lip}, this implies that 
\[t\mapsto \int_{\mathbb R^d}\varphi(x)\,d\mu_t(x)\]
is Lipschitz continuous. Hence its distributional derivative is in $L^{\infty}$ and coincides with the pointwise derivative almost everywhere.
Thus, in the sense of distributions in $]0,T[$, we obtain for all $\varphi\in C^1_c(\mathbb R^d)$ 
\begin{align*}
\dfrac{d}{dt}\int_{\mathbb R^d}\varphi(x)\,d\mu_t(x)
=&\dfrac{d}{dt}\int_{\mathbb R^d\times\Gamma_T}\varphi(\gamma(t))\,d\boldsymbol\eta(x,\gamma)=\int_{\mathbb R^d\times\Gamma_T}\langle \nabla\varphi(\gamma(t)),\dot\gamma(t)\rangle\,d\boldsymbol\eta(x,\gamma)\\
=&\int_{\mathbb R^d}\langle\nabla\varphi(y), \int_{e^{-1}_t(y)}\dot\gamma(t)\,d\eta_{t,y}(x,\gamma)\rangle\,d\mu_t(y),
\end{align*}
where we disintegrated $\boldsymbol\eta$ w.r.t. $e_t$ obtaining $\boldsymbol\eta=\mu_t\otimes\eta_{t,y}$ and used the fact that $\|\nabla\varphi\|_{\infty}$ is bounded, and 
that the map $\gamma\mapsto\|\dot\gamma\|_{L^{\infty}}$ is in $L^1_{\boldsymbol\eta}$ due to the uniform bound on the moments.
By Jensen's inequality, we have 
\[d_{F(\mu_t,y)}\left(\int_{e^{-1}_t(y)}\dot\gamma(t)\,d\eta_{t,y}(x,\gamma)\right)\le \int_{e^{-1}_t(y)}d_{F(\mu_t,y)}\left(\dot\gamma(t)\right)\,d\eta_{t,y}(x,\gamma)=0,\]
and so for $\mu_t$-a.e. $y\in\mathbb R^d$ and a.e. $t\in [0,T]$ we have \[v_t(y):= \int_{e^{-1}_t(y)}\dot\gamma(t)\,d\eta_{t,y}(x,\gamma)\in F(\mu_t,y),\] hence
$\boldsymbol\mu=\{\mu_t\}_{t\in [0,T]}$ is an admissible trajectory driven by $\boldsymbol\nu=\{\nu_t\}_{t\in [0,T]}$ with $\nu_t=v_t\mu_t$ for a.e. $t\in [0,T]$.
\item[]
\item \emph{Necessity}. Assume that $\boldsymbol\mu=\{\mu_t\}_{t\in [0,T]}$ is an admissible trajectory driven by $\boldsymbol\nu=\{\nu_t\}_{t\in[0,T]}$. 
Set $v_t(x)=\dfrac{\nu_t}{\mu_t}(x)\in F(\mu_t,x)$ for $\mu_t$-a.e. $x\in\mathbb R^d$ and a.e. $t\in [0,T]$.
Filippov's Theorem (see e.g. \cite[Theorem 8.2.10]{AuF}) implies that there exists a Borel selection $\xi(\cdot)$ of $F(\delta_0,0)$, such that 
\[|v_t(x)-\xi(x)|=d_{F(\delta_0,0)}(v_t(x))\]
for all $x\in\mathbb R^d$, and so we have
\begin{align*}
\left(\int_0^T\int_{\mathbb R^d}|v_t(x)|^p\,d\mu_t(x)\,dt\right)^{1/p}\le&\int_0^T\left(\int_{\mathbb R^d}|v_t(x)-\xi(x)|^p\,d\mu_t(x)\right)^{1/p}\,dt+\\
&+\int_0^T\left(\int_{\mathbb R^d}|\xi(x)|^p\,d\mu_t(x)\right)^{1/p}\,dt\\
=&\int_0^T\left(\int_{\mathbb R^d}d^p_{F(\delta_0,0)}(v_t(x))\,d\mu_t(x)\right)^{1/p}\,dt+TK_F\\
\le&2^{p-1}L\int_0^T\left[W_p(\delta_0,\mu_t)+\mathrm{m}_p^{1/p}(\mu_t)\right]\,dt+TK_F\\
\le&2^pL\int_0^T\mathrm{m}^{1/p}_p(\mu_t)\,dt+TK_F<+\infty.
\end{align*}
By \cite[Theorem 8.2.1]{AGS}, there exists $\boldsymbol\eta\in\mathscr P(\mathbb R^d\times\Gamma_T)$ such that $\mu_t=e_t\sharp\boldsymbol\eta$ for all $t\in [0,T]$ and $\boldsymbol\eta$ is
concentrated on $(x,\gamma)\in\mathbb R^d\times\Gamma_T$ with $\gamma\in AC([0,T])$, $\dot\gamma(t)=v_t(\gamma(t))\in F(\mu_t,\gamma(t))$ for a.e. $t\in [0,T]$ and $\gamma(0)=x$.
Thus $\boldsymbol\mu\in \Upsilon_F(\mu_0,\boldsymbol\mu)$.
\end{enumerate}
\end{proof}

\begin{remark}\label{rmk:SPLp}
Notice that by Definition of $\Upsilon_F$ in Definition \ref{def:bas}, Proposition \ref{prop:SP} gives in fact a Superposition Principle along the line of \cite[Theorem 8.2.1]{AGS} adapted to nonlocal differential inclusions. Indeed, under the given assumptions, it states that $\boldsymbol\mu=\{\mu_t\}_{t\in [0,T]}$ is an admissible trajectory if and only if
there exists $\boldsymbol\eta\in\mathscr P(\mathbb R^d\times\Gamma_T)$ such that
\begin{itemize}
\item $\mu_t=e_t\sharp\boldsymbol\eta$ for all $t\in [0,T]$, $\mu_{t=0}=\mu_0$;
\item for $\boldsymbol\eta$-a.e. $(x,\gamma)\in\mathbb R^d\times\Gamma_T$ and a.e. $t\in[0,T]$ it holds $\gamma\in AC([0,T])$, $\gamma(0)=x$, $\dot\gamma(t)\in F(\mu_t,\gamma(t))$ .
\end{itemize}
In this case, we say that $\boldsymbol\mu$ is \emph{represented} by $\boldsymbol\eta$.
\end{remark}

\begin{corollary}[Existence of admissible trajectories]\label{cor:exist}
Assume Hypothesis \ref{HP} for $F$ and let $T>0$. The set-valued map $\mathcal A:\mathscr P_p(\mathbb R^d)\rightrightarrows C^0([0,T];\mathscr P_p(\mathbb R^d))$ 
is upper semicontinuous with nonempty compact images.
\end{corollary}
\begin{proof}
Let $\mu\in\mathscr P_p(\mathbb R^d)$. Given $R>0$, define
\begin{multline*}\mathcal C(R):=\Big\{\boldsymbol\mu=\{\mu_t\}_{t\in [0,T]}\subseteq\mathscr P_p(\mathbb R^d):\mu_0=\mu,\textrm{ and for all $t,s\in [0,T]$ }\mathrm{m}_p^{1/p}(\mu_t)\le R, \\
W_p(\mu_t,\mu_s)\le\left(K_F+LR+Le^{LT}\left(\mathrm{m}_p^{1/p}(\mu)+K_FT+LTR\right)\right)\cdot|t-s|\Big\}.\end{multline*}
Recalling that the concatenation of solutions of the continuity equation is again a solution of the continuity equation driven by the time concatenation of the vector fields (see \cite[Lemma 4.4]{DNS}),
in order to prove that $\mathcal A(\mu)\ne \emptyset$ it is not restrictive to assume $LT<1/2$.
In particular, we have $1-e^{LT}LT>0$. Define
\[R:=\dfrac{e^{LT}(\mathrm{m}^{1/p}_p(\mu)+K_FT)}{1-e^{LT}LT}\ge \mathrm{m}^{1/p}_p(\mu).\]
Notice that $\mathcal C(R)\ne\emptyset$, since it contains the constant curve $\mu_t\equiv\mu$ for all $t\in[0,T]$, it is convex, and it is compact in $C^0([0,T];\mathscr P_p(\mathbb R^d))$ by Ascoli-Arzel\`a theorem.
Moreover, defining $\Upsilon_F(\cdot,\cdot)$ as in Definition \ref{def:bas}, we have $\Upsilon_F(\mu,\mathcal C(R))\subseteq \mathcal C(R)$ by Lemma \ref{lemma:moment} and the choice of $R$.
By Kakutani-Ky Fan Theorem (see e.g. \cite[Theorem 1]{KF}) we have that there exists $\boldsymbol\mu\in \mathcal C(R)$ such that $\boldsymbol\mu\in\Upsilon_F(\mu,\boldsymbol\mu)$, i.e.,
by Proposition \ref{prop:SP}, $\boldsymbol\mu$ is an admissible trajectory starting from $\mu$.
All the other properties of $\mathcal A(\cdot)$ trivially follows from the fact that $\Upsilon_F(\cdot)$ is upper semicontinuous with nonempty compact images.
\end{proof}

\begin{remark}
An alternative proof of existence of admissible trajectories, i.e. $\mathcal A(\mu)\neq\emptyset$, can be found for example in \cite[Theorem 6.1]{orrieri} where the author provides sufficient conditions in order to ensure existence (and uniqueness) of solutions of a continuity equation for some given non-local vector field.
\end{remark}

\begin{theorem}[Filippov-type estimate for the set of admissible trajectories]\label{thm:filippov}
Assume Hypothesis \ref{HP} for $F$. Let $T>0$, $\mu^{(A)},\mu^{(B)}\in\mathscr P_p(\mathbb R^d)$ be given. Let $\boldsymbol{\mu^{(A)}}=\{\mu^{(A)}_t\}_{t\in [0,T]}$ be an admissible trajectory satisfying $\mu^{(A)}_0=\mu^{(A)}$.\par
Then there exists an admissible trajectory $\boldsymbol{\mu^{(B)}}=\{\mu^{(B)}_t\}_{t\in [0,T]}$ satisfying $\mu^{(B)}_0=\mu^{(B)}$
such that
\[W_p(\mu^{(A)}_t,\mu^{(B)}_t)\le 2^{\frac{p-1}{p}}e^{L(2+Le^{LT})T}\cdot W_p(\mu^{(A)},\mu^{(B)})\quad \textrm{for all }t\in [0,T].\]
In particular, the set-valued map $\mathcal A:\mathscr P_p(\mathbb R^d)\rightrightarrows C^0([0,T];\mathscr P_p(\mathbb R^d))$ is Lipschitz continuous.
\end{theorem}
\begin{proof}
Let $\boldsymbol\pi\in\Pi^p_o(\mu^{(A)},\mu^{(B)})$ be an optimal transport plan between $\mu^{(A)}$ and $\mu^{(B)}$ for the $p$-Wasserstein distance. By disintegrating $\boldsymbol\pi$ w.r.t. $\mathrm{pr}_1:\mathbb R^d\times\mathbb R^d\to\mathbb R^d$, defined by $\mathrm{pr}_1(x,y)=x$, we have a Borel collection 
of measures $\{\pi_x\}_{x\in\mathbb R^d}\subseteq\mathscr P(\mathbb R^d\times\mathbb R^d)$, uniquely defined for $\mu^{(A)}$-a.e. $x\in\mathbb R^d$,
such that $\boldsymbol\pi=\mu^{(A)}\otimes \pi_x$.
\par\medskip\par
According to Proposition \ref{prop:SP}, there exists $\boldsymbol{\eta^{(A)}}\in\mathscr P(\mathbb R^d\times\Gamma_T)$ 
concentrated on pairs $(x,\gamma)\in\mathbb R^d\times\Gamma_T$ with $\gamma\in AC([0,T])$, $\dot\gamma(t)\in F(e_t\sharp\boldsymbol{\eta^{(A)}},\gamma(t))$ for a.e. $t\in [0,T]$ and $\gamma(0)=x$
such that $\mu_t^{(A)}=e_t\sharp\boldsymbol{\eta^{(A)}}$. 
\par\medskip\par
Let $\boldsymbol\theta\in C^0([0,T];\mathscr P_p(\mathbb R^d))$, and define the set-valued map $S^{\boldsymbol\theta}(\cdot)$ as in Definition \ref{def:bas}.
Define the set-valued map $R^{\boldsymbol\theta}:\mathbb R^d\times\mathrm{supp}\,\boldsymbol{\eta^{(A)}}\rightrightarrows\Gamma_T$ by
\begin{align*}R^{\boldsymbol\theta}(y,x,\gamma):=\left\{\xi\in S^{\boldsymbol\theta}(y):\begin{array}{l}|\gamma(t)-\xi(t)|\le e^{LT}|\gamma(0)-\xi(0)|+L(e^{LT}+1)\int_0^t W_p(\mu_\tau^{(A)},\theta_\tau)\,d\tau\\ \textrm{ for all }t\in[0,T]\end{array}\right\}.\end{align*}
Notice that this map has closed domain, closed graph, and compact values since $R^{\boldsymbol\theta}(y,x,\gamma)\subseteq S^{\boldsymbol\theta}(y)$, thus it is upper semicontinuous, hence Borel measurable.
\par\medskip\par
We prove that it has nonempty images. Given a point $(y,x,\gamma)\in \mathbb R^d\times\mathrm{supp}\,\boldsymbol{\eta^{(A)}}$, there are sequences $\{x_n\}_{n\in\mathbb N}$ converging to $x$ and $\{\gamma_n\}_{n\in\mathbb N}\subseteq AC([0,T])$ uniformly converging to $\gamma$ such that $x_n=\gamma_n(0)$ and $\dot \gamma_n(t)\in F(\mu_t^{(A)},\gamma_n(t))$ for a.e. $t\in[0,T]$.
According to Filippov's theorem (see \cite[Theorem 10.4.1]{AuF}), for every $n\in\mathbb N$ there exists $\xi_n\in S^{\boldsymbol\theta}(y)$
such that
\begin{align*}
|\gamma_n(t)-\xi_n(t)|\le&e^{LT}|\gamma_n(0)-\xi_n(0)|+(Le^{LT}+1)\int_0^t d_{F(\theta_t,\gamma_n(\tau))}(\dot\gamma_n(\tau))\,d\tau\\
\le&e^{LT}|\gamma_n(0)-\xi_n(0)|+L(Le^{LT}+1)\int_0^t W_p(\theta_t,\mu_t^{(A)})\,d\tau,
\end{align*}
recalling the Lipschitz continuity of $F(\cdot)$ and the choice of $\gamma_n$.
By compactness of $S^{\boldsymbol\theta}(y)$, up to passing to a subsequence, we may assume that $\{\xi_n\}_{n\in\mathbb N}$ uniformly converges
to $\xi\in S^{\boldsymbol\theta}(y)$ and, by construction, we have $\xi\in R^{\boldsymbol\theta}(y,x,\gamma)$, hence $R^{\boldsymbol\theta}(\cdot)$ is Borel measurable with closed domain and nonempty images, thus it admits a Borel selection $h_{\boldsymbol\theta}:\mathbb R^d\times\mathrm{supp}\,\boldsymbol{\eta^{(A)}}\to \Gamma_T$.
We extend $h_{\boldsymbol\theta}(\cdot)$ to a Borel map defined on the whole of $\mathbb R^d\times\mathbb R^d\times\Gamma_T\to \Gamma_T$ by setting
$h_{\boldsymbol\theta}(y,x,\gamma)=\gamma$ if $(x,\gamma)\notin\mathrm{supp}\,\boldsymbol{\eta^{(A)}}$.
\par\medskip\par
Define $\boldsymbol{\eta^\theta}\in\mathscr P(\mathbb R^d\times\Gamma_T)$ by 
\[\int_{\mathbb R^d\times\Gamma_T}\varphi(y,\xi)\,d\boldsymbol{\eta^{\theta}}(y,\xi)=\int_{\mathbb R^d\times\Gamma_T}\left[\int_{\mathbb R^d\times\mathbb R^d}\varphi(x,h_{\boldsymbol\theta}(y,x,\gamma))\,d\pi_x(x,y)\right]\,d\boldsymbol{\eta^{(A)}}(x,\gamma),\]
and set $\boldsymbol{\mu^\theta}=\{\mu^\theta_t\}_{t\in [0,T]}$ where $\mu^\theta_t=e_t\sharp\boldsymbol{\eta^\theta}$ for all $t\in [0,T]$.
\par\medskip\par
We have, by construction, 
\[\mathrm{supp}\,\boldsymbol{\eta^{\theta}}\subseteq \left\{(y,\xi)\in\mathbb R^d\times\Gamma_T:\,\xi\in S^{\boldsymbol\theta}(y)\right\}.\]
Notice that
\begin{align*}
\int_{\mathbb R^d}\varphi(x)\,d\mu^\theta_0(x)=&\int_{\mathbb R^d\times\Gamma_T}\int_{\mathbb R^d\times\mathbb R^d}\varphi\left(h_{\boldsymbol\theta}(y,x,\gamma)(0)\right)\,d\pi_x(x,y)\,d\boldsymbol{\eta^{(A)}}(x,\gamma)\\
=&\int_{\mathbb R^d}\int_{\mathbb R^d\times\mathbb R^d}\varphi\left(y\right)\,d\pi_x(x,y)\,d\mu^{(A)}(x)\\
=&\int_{\mathbb R^d\times\mathbb R^d}\varphi\left(y\right)\,d\boldsymbol\pi(x,y)=\int_{\mathbb R^d\times\mathbb R^d}\varphi\left(y\right)\,d\mu^{(B)}(y).
\end{align*}
Thus $\boldsymbol{\mu^\theta}\in \Upsilon_F(\mu^{(B)},\boldsymbol\theta)$, where $\Upsilon_F(\cdot,\cdot)$ is as in Definition \ref{def:bas}.
\par\medskip\par
We have
\begin{align*}
&W_p(\mu_t^{(A)},\mu^{\theta}_t)\le\left(\int_{\mathbb R^d\times\Gamma_T}\int_{\mathbb R^d\times\mathbb R^d}|\gamma(t)-h_{\boldsymbol\theta}(y,x,\gamma)(t)|^p\,d\pi_x(x,y)\,d\boldsymbol{\eta^{(A)}}(x,\gamma)\right)^{1/p}\\
&\le\left(\int_{\mathbb R^d\times\Gamma_T}\int_{\mathbb R^d\times\mathbb R^d}\left[e^{LT}|x-y|+L(Le^{LT}+1)\int_0^t W_p(\mu_\tau^{(A)},\theta_\tau)\,d\tau\right]^p\,d\pi_x(x,y)\,d\boldsymbol{\eta^{(A)}}(x,\gamma)\right)^{1/p}\\
&\le2^{\frac{p-1}{p}}\left[e^{LT}W_p(\mu^{(A)},\mu^{(B)})+L(Le^{LT}+1)\int_0^t W_p(\mu_\tau^{(A)},\theta_\tau)\,d\tau\right].
\end{align*}
Thus, since $W_p(\mu_t^{(A)},\mu^{\theta}_t)\ge W_p(\delta_0,\mu^{\theta}_t)-W_p(\mu_t^{(A)},\delta_0)=\mathrm{m}_p^{1/p}(\mu^\theta_t)-\mathrm{m}_p^{1/p}(\mu_t^{(A)})$, we have
\begin{align*}
\mathrm{m}_p^{1/p}(\mu^\theta_t)\le&\mathrm{m}_p^{1/p}(\mu_t^{(A)})+2^{\frac{p-1}{p}}\left[e^{LT}W_p(\mu^{(A)},\mu^{(B)})+LD\int_0^t \mathrm{m}^{1/p}_p(\mu_\tau^{(A)})\,d\tau+LD\int_0^t\mathrm{m}_p^{1/p}(\theta_\tau)\,d\tau\right]\\
\le&(1+LTD)\sup_{t\in[0,T]}\mathrm{m}_p^{1/p}(\mu_t^{(A)})+2^{\frac{p-1}{p}}\left[e^{LT}W_p(\mu^{(A)},\mu^{(B)})+LD\int_0^t\mathrm{m}_p^{1/p}(\theta_\tau)\,d\tau\right],
\end{align*}
where we denoted with $D=Le^{LT}+1$.
As in the proof of Corollary \ref{cor:exist},  without loss of generality we can assume that $0\le 2^{\frac{p-1}{p}}LDT<1$. The general case will follow by concatenating finitely many pieces of admissible curves defined on time-subintervals of sufficiently small length.
We take $R>0$ sufficiently large such that
\[R\ge \dfrac{\displaystyle(1+LTD)\sup_{t\in[0,T]}\mathrm{m}_p^{1/p}(\mu_t^{(A)})+2^{\frac{p-1}{p}}e^{LT}W_p(\mu^{(A)},\mu^{(B)})}{1-2^{\frac{p-1}{p}}LDT}\ge\mathrm{m}_p^{1/p}(\mu^{(B)}) ,\]
and such that $\mathrm{m}_p^{1/p}(\theta_t)\le R$ for all $t\in [0,T]$ and also $\mathrm{m}_p^{1/p}(\mu^\theta_t)\le R$
for all $t\in [0,T]$.
Define a sequence $\{\boldsymbol\mu^{(n)}=\{\mu^{(n)}_t\}_{t\in [0,T]}\}_{n\in\mathbb N}\subseteq C^0([0,T];\mathscr P_p(\mathbb R^d))$ by setting
$\boldsymbol\mu^{(0)}$ to be the constant $\mu^{(B)}$ and $\boldsymbol\mu^{(n)}$ to be equal to $\boldsymbol{\mu^{\theta}}$
with $\boldsymbol\theta=\boldsymbol\mu^{(n-1)}$. Notice that $\mu^{(n)}_0=\mu^{(B)}$ for all $n\in\mathbb N$.
According to Lemma \ref{lemma:moment}, the family $\{\boldsymbol\mu^{(n)}\}_{n\in\mathbb N}$ is relatively compact, thus
up to passing to a subsequence, we may assume that it converges to $\boldsymbol\mu^{\infty}\in C^0([0,T];\mathscr P_p(\mathbb R^d))$. Since $\boldsymbol\mu^{(n)}\in \Upsilon_F(\mu^{(B)},\boldsymbol\mu^{(n-1)})$, by recalling the u.s.c.
of $\Upsilon_F(\cdot,\cdot)$ proved in Proposition \ref{prop:uscsol}, we obtain that $\boldsymbol\mu^{\infty}\in \Upsilon_F(\mu^{(B)},\boldsymbol\mu^{\infty})$, i.e.,
$\boldsymbol{\mu^\infty}$ is an admissible trajectory, starting from $\mu^{(B)}$.
Finally, by passing to the limit in
\begin{align*}
W_p(\mu_t^{(A)},\mu^{(n)}_t)\le& 2^{\frac{p-1}{p}}\left[e^{LT}W_p(\mu^{(A)},\mu^{(B)})+LD\int_0^t W_p(\mu_\tau^{(A)},\mu^{(n-1)}_\tau)\,d\tau\right],
\end{align*}
we have
\begin{align*}
W_p(\mu_t^{(A)},\mu^{\infty}_t)\le& 2^{\frac{p-1}{p}}\left[e^{LT}W_p(\mu^{(A)},\mu^{(B)})+LD\int_0^t W_p(\mu_\tau^{(A)},\mu^{\infty}_\tau)\,d\tau\right],
\end{align*}
and, by Gr\"onwall's Lemma,
\begin{align*}
W_p(\mu_t^{(A)},\mu^{\infty}_t)\le& 2^{\frac{p-1}{p}}e^{\hat D T} W_p(\mu^{(A)},\mu^{(B)}),
\end{align*}
as desired, where $\hat D=L(2+Le^{LT})$. The proof is concluded by setting $\boldsymbol{\mu^{(B)}}=\boldsymbol{\mu^{\infty}}$. The last assertion
trivially follows.
\end{proof}

\begin{lemma}[Initial velocity set]\label{lemma:invelset}
Assume Hypothesis \ref{HP} for $F$ and let $\Xi(\cdot,\cdot)$ be as in Definition \ref{def:bas}. Let $\mu\in\mathscr P_p(\mathbb R^d)$.
\begin{enumerate}
\item Given any Borel selection $v_\mu:\mathbb R^d\to\mathbb R^d$ of $F(\mu,\cdot)$, there exists $\boldsymbol\eta\in\mathscr P(\mathbb R^d\times\Gamma_T)$ such that, set $\mu_t=e_t\sharp\boldsymbol\eta$ for $t\in [0,T]$,
we have $\boldsymbol\mu=\{\mu_t\}_{t\in [0,T]}\in \mathcal A^p_{[0,T]}(\mu)$, $\boldsymbol\eta\in \Xi(\mu,\boldsymbol\mu)$ and
\[\left|\dfrac{\gamma(t)-\gamma(0)}{t}-v_\mu(x)\right|\le Le^{Lt}\left[\dfrac{1}{t}\int_0^tW_p(\mu_\tau,\mu)\,d\tau+\frac{t}{2}|v_\mu(x)|\right]\]
for $\boldsymbol\eta$-a.e. $(x,\gamma)\in\mathbb R^d\times\Gamma_T$.
\item Given any admissible trajectory $\boldsymbol\mu=\{\mu_t\}_{t\in [0,T]}\in \mathcal A^p_{[0,T]}(\mu)$, there exists $\boldsymbol\eta\in \Xi(\mu,\boldsymbol\mu)$ such that
$\mu_t=e_t\sharp\boldsymbol\eta$ and for $\boldsymbol\eta$-a.e. $(x,\gamma)\in\mathbb R^d\times\Gamma_T$ we have
\[\lim_{t\to 0^+}\int_{\mathbb R^d\times\Gamma_T}d^p_{F(\mu,x)}\left(\dfrac{\gamma(t)-\gamma(0)}{t}\right)\,d\boldsymbol\eta(x,\gamma)=0.\]
\end{enumerate}
\end{lemma}
\begin{proof}$ $\par
\textbf{1. }Without loss of generality, we may assume $LT\le 1/2$, the general case will be obtained concatenating $\boldsymbol\mu$ with any other admissible trajectory starting from $\mu_T$.
Let $v_0:\mathbb R^d\to\mathbb R^d$ be any Borel selection of $F(\mu,\cdot)$.
Define $\gamma_x:[0,T]\to\mathbb R^d$ by $\gamma_x(t)=x+v_0(x)\cdot t$, and observe that $x\mapsto \gamma_x$ is a Borel map.
Let $\boldsymbol\theta=\{\theta_t\}_{t\in [0,T]}\in C^0([0,T];\mathscr P_p(\mathbb R^d))$ such that $\theta_0=\mu$, and notice that
\[d_{F(\theta_t,\gamma_x(t))}(\dot\gamma_x(t))\le L\left[W_p(\theta_t,\mu)+t|v_0(x)|\right].\]
Thus, by Filippov's Theorem (see \cite[Theorem 10.4.1]{AuF}) the set-valued map $R^{\boldsymbol\theta}:\mathbb R^d\rightrightarrows \Gamma_T$ defined as
\[R^{\boldsymbol\theta}(x):=\left\{\xi\in S^{\boldsymbol\theta}(x):\, \begin{array}{l}|\gamma_x(t)-\xi(t)|\le Le^{Lt}\int_0^t\left[W_p(\theta_\tau,\mu)+\tau|v_0(x)|\right]\,d\tau\\\textrm{for all }t\in[0,T]\end{array}\right\}\]
has nonempty images for every $x\in\mathbb R^d$. Notice that this set-valued map has closed images and it is Borel measurable by \cite[Theorem 8.2.9]{AuF}, thus it admits a Borel selection $h_{\boldsymbol\theta}:\mathbb R^d\to \Gamma_T$.
Set $\boldsymbol{\eta^\theta}=\mu\otimes\delta_{h_{\boldsymbol\theta}(x)}$ and $\boldsymbol{\mu^\theta}=\{\mu^\theta_t\}_{t\in [0,T]}$, $\mu^\theta_t=e_t\sharp\boldsymbol{\eta^\theta}$.
By construction we have $\boldsymbol{\mu^\theta}\in \Upsilon_F(\mu,\boldsymbol\theta)$, moreover for all $x\in\mathbb R^d$ we have
\begin{align*}
\mathrm{m}_p^{1/p}(\mu^\theta_t)=&\left(\int_{\mathbb R^d}\left|h_{\boldsymbol\theta}(x)(t)\right|^p\,d\mu(x)\right)^{1/p}\\
\le&\left(\int_{\mathbb R^d}\left|h_{\boldsymbol\theta}(x)(t)-\gamma_x(t)\right|^p\,d\mu(x)\right)^{1/p}+\left(\int_{\mathbb R^d}\left|\gamma_x(t)\right|^p\,d\mu(x)\right)^{1/p}\\
\le&Le^{Lt}\left[\int_{\mathbb R^d}\left|\int_0^t\left(W_p(\theta_\tau,\mu)+\tau|v_0(x)|\right)\,d\tau\right|^p\,d\mu\right]^{1/p}+\mathrm{m}_p^{1/p}(\mu)+t\|v_0\|_{L^p_\mu}\\
\le&Le^{Lt}\int_0^tW_p(\theta_\tau,\mu)\,d\tau+\mathrm{m}_p^{1/p}(\mu)+(LTe^{Lt}+t)\|v_0\|_{L^p_\mu}\\
\le&Le^{Lt}\int_0^t\mathrm{m}_p^{1/p}(\theta_\tau)\,d\tau+(1+Le^{Lt}t)\mathrm{m}_p^{1/p}(\mu)+(LTe^{Lt}+t)\|v_0\|_{L^p_\mu}.
\end{align*}
Furthermore,
\[\left|\dfrac{h_{\boldsymbol\theta}(x)(t)-h_{\boldsymbol\theta}(0)}{t}-v_0(x)\right|=\left|\dfrac{h_{\boldsymbol\theta}(x)(t)-\gamma_x(t)}{t}\right|\le Le^{Lt}\dfrac{1}{t}\int_0^t\left[W_p(\theta_\tau,\mu)+\tau|v_0(x)|\right]\,d\tau.\]
Choose  
\[R\ge\dfrac{(1+LTe^{LT})\mathrm{m}_p^{1/p}(\mu)+(LTe^{LT}+T)\|v_0\|_{L^p_\mu}}{1-LTe^{LT}}\ge \mathrm{m}_p^{1/p}(\mu),\]
and notice that if $\mathrm{m}_p^{1/p}(\theta_t)\le R$ for all $t\in [0,T]$, then $\mathrm{m}_p^{1/p}(\mu^\theta_t)\le R$ for all $t\in [0,T]$.
Define sequences $\{\boldsymbol\mu^{(n)}=\{\mu^{(n)}_t\}_{t\in [0,T]}\}_{n\in\mathbb N}\subseteq C^0([0,T];\mathscr P_p(\mathbb R^d))$ and $\{\boldsymbol\eta^{(n)}\}_{n\in\mathbb N}\subseteq \mathscr P(\mathbb R^d\times\Gamma_T)$
by setting $\boldsymbol\mu^{(0)}$ to be the constant $\mu$, $\boldsymbol\mu^{(n)}$ and $\boldsymbol{\eta^{(n)}}$ to be equal to $\boldsymbol{\mu^{\theta}}$ and $\boldsymbol{\eta^{\theta}}$, respectively,
with $\boldsymbol\theta=\boldsymbol\mu^{(n-1)}$ for all $n\in\mathbb N$.
According to Lemma \ref{lemma:moment}, the families $\{\boldsymbol\mu^{(n)}\}_{n\in\mathbb N}$ and $\{\boldsymbol{\eta^{(n)}}\}_{n\in\mathbb N}$ are relatively compact, thus
up to passing to a subsequence, we may assume that the sequences converge to $\boldsymbol\mu^{\infty}=\{\mu^\infty_t\}_{t\in[0,T]}\in C^0([0,T];\mathscr P_p(\mathbb R^d))$ and to $\boldsymbol{\eta^{\infty}}\in\mathscr P(\mathbb R^d\times\Gamma_T)$,
with $\mu^{\infty}_t=e_t\sharp\boldsymbol{\eta^\infty}$ for all $t\in[0,T]$.
Since $\boldsymbol\mu^{(n)}\in \Upsilon_F(\mu,\boldsymbol\mu^{(n-1)})$ ($\Upsilon_F(\cdot,\cdot)$ defined in Definition \ref{def:bas}), by recalling the u.s.c.
of $\Upsilon_F(\cdot,\cdot)$ proved in Proposition \ref{prop:uscsol}, we obtain that $\boldsymbol\mu^{\infty}\in \Upsilon_F(\mu,\boldsymbol\mu^{\infty})$, i.e.,
$\boldsymbol{\mu^\infty}$ is an admissible trajectory, starting from $\mu$.
Recall that for $\boldsymbol{\eta^{\infty}}$-a.e. $(x,\gamma)$ there exists a sequence $\{(x_n,\xi_n)\}_{n\in\mathbb N}\subseteq\mathbb R^d\times\Gamma_T$ converging to $(x,\gamma)$ such that $(x_n,\xi_n)\in\mathrm{supp}\,\boldsymbol{\eta^{(n)}}$.
Thus, without loss of generality, we may assume for all $t\in[0,T]$
\[\left|\dfrac{\xi_n(t)-\xi_n(0)}{t}-v_0(x)\right|\le Le^{Lt}\dfrac{1}{t}\int_0^t\left[W_p(\mu^{(n-1)}_\tau,\mu)+\tau|v_0(x)|\right]\,d\tau,\]
and, by passing to the limit,
\[\left|\dfrac{\gamma(t)-\gamma(0)}{t}-v_0(x)\right|\le Le^{Lt}\dfrac{1}{t}\int_0^t\left[W_p(\mu^\infty_\tau,\mu)+\tau|v_0(x)|\right]\,d\tau.\]
\textbf{2. }Recall the Superposition Principle in Proposition \ref{prop:SP} and the definitions of $\Upsilon_F$ and $\Xi$ in Definition \ref{def:bas}. Then, from the assumption it follows that there exists $\boldsymbol\eta\in\Xi(\mu,\boldsymbol\mu)$ such that $\mu_t=e_t\sharp\boldsymbol\eta$ for all $t\in [0,T]$.
By Jensen's inequality, for $\boldsymbol\eta$-a.e. $(x,\gamma)\in\mathbb R^d\times\Gamma_T$, we have
\begin{align*}
d_{F(\mu,x)}\left(\dfrac{\gamma(t)-\gamma(0)}{t}\right)=&d_{F(\mu,x)}\left(\dfrac{1}{t}\int_{0}^t \dot\gamma(s)\,ds\right)\le
\dfrac{1}{t}\int_{0}^t d_{F(\mu,x)}\left(\dot\gamma(s)\right)\,ds\\ 
\le&\dfrac{L}{t}\int_{0}^t \left(W_p(\mu_s,\mu)+|\gamma(s)-x|\right)\,ds\\ 
\le&\dfrac{L}{t}\int_{0}^t \left(W_p(\mu_s,\mu)+\|\dot\gamma\|_{L^{\infty}_{\boldsymbol\eta}}s\right)\,ds.
\end{align*}
We conclude by taking the $L^p_{\boldsymbol\eta}$-norm and using Lemma \ref{lemma:moment}.
\end{proof}

\section{Generalized targets}\label{sec:gentarget}

In this section, we provide the generalized notion of target set in the space of probability measures, thus extending in a natural way the classical concept of target set in $\mathbb R^d$.
A naive physical interpretation of the generalized target can be given as follows: to describe the state of the system, an observer chooses to measure 
some quantities $\phi$. The results of the measurements are the \emph{averages} of the quantities $\phi$ with respect to the measure $\mu_t$, representing the state of
the system at time $t$. Our aim is to steer the system to states where the result of such measurements is below a fixed threshold (that, without loss of generality, we assume to be $0$). 
The following result provides a characterization of the class of such generalized target. 

\begin{lemma}
Let $\tilde S\subseteq \mathscr P(\mathbb R^d)$ be nonempty.
Then, $\tilde S$ is $w^*$-closed and convex if and only if there exists a family $\Phi\subseteq C^0_b(\mathbb R^d)$ such that $\tilde S$ can be written as follows
\begin{equation}\label{eq:defSgen}
\tilde S=\left\{\mu\in\mathscr P(\mathbb R^d):\,\int_{\mathbb R^d}\varphi(x)\,d\mu(x)\le 0\textrm{ for all }\varphi\in\Phi\right\}.
\end{equation}
\end{lemma}
\begin{proof}
We first prove the necessity, so let $\tilde S$ be as in \eqref{eq:defSgen} for some fixed $\Phi\subseteq C^0_b(\mathbb R^d)$. Then, the convexity of $\tilde S$ comes by linearity of the integral w.r.t. the measure, while the closure in $w^*$ topology follows immediately since $\Phi$ is a family of test functions for $w^*$-convergence.\par
We pass to the proof of the sufficiency.
Recalling formula (5.1.7) in \cite[Remark 5.1.2]{AGS}, we have that $\bar \mu\in\tilde S$ if and only if for all $\psi\in C^0_b(\mathbb R^d)$ it holds
\[\int_{\mathbb R^d}\psi(x)\,d\bar\mu(x)\le \sup_{\mu\in\tilde S}\int_{\mathbb R^d}\psi(x)\,d\mu(x).\]
Given $\psi\in C^0_b(\mathbb R^d)$, set 
\[C_{\psi}:=\sup_{\mu\in\tilde S}\int_{\mathbb R^d}\psi(x)\,d\mu(x)\le +\infty.\]
Then we have that $\bar \mu\in\tilde S$ if and only if for all $\psi\in C^0_b(\mathbb R^d)$ such that $C_{\psi}<+\infty$ it holds
\[\int_{\mathbb R^d}\left[\psi(x)-C_{\psi}\right]\,d\bar\mu(x)\le 0.\]
Then, to get \eqref{eq:defSgen} it sufficies to take
\[\Phi:=\left\{\varphi:=\psi-C_{\psi}:\, \psi\in C^0_b(\mathbb R^d)\textrm{ and }C_{\psi}<+\infty\right\}.\]
\end{proof}

\begin{definition}[Generalized targets]\label{def:gentar}
Let $\tilde S\subseteq\mathscr P(\mathbb R^d)$ be nonempty $w^*$-closed and convex, $\Phi\subseteq C^0_b(\mathbb R^d)$.
We say that $\tilde S$ is a \emph{generalized target generated by} $\Phi$, and write $\tilde S=\tilde S^\Phi$ if  
\begin{equation}\label{eq:gentar}\tilde S:=\left\{\mu\in\mathscr P(\mathbb R^d):\,\int_{\mathbb R^d}\varphi(x)\,d\mu(x)\le 0\textrm{ for all }\varphi\in\Phi\right\}.\end{equation}
Given $p\ge 1$ we set $\tilde S^\Phi_p=\tilde S^\Phi\cap \mathscr P_p(\mathbb R^d)$, and we define the \emph{generalized distance} from 
$\tilde S_p^{\Phi}$ to be the $1$-Lipschitz continuous map given by $\displaystyle\tilde d_{\tilde S_p^\Phi}(\cdot):=\inf_{\mu\in \tilde S_p^\Phi}W_p(\cdot,\mu)$.
\end{definition}

\begin{remark}\label{rem:smoothphi}
\begin{itemize}\item[]
\item In Definition \ref{def:gentar} we can equivalently assume that $\Phi$ is a set of continuous bounded functions, or bounded Lipschitz functions, or even just l.s.c. functions bounded from below. 
Moreover, without loss of generality, we can always assume that $\Phi$ is convex.
Indeed, assume that $\Psi$ is a set of l.s.c. functions bounded from below. For all $\psi\in \Psi$ and $k\in\mathbb N\setminus\{0\}$ we define
a Lipschitz continuous bounded map $\varphi_{k}^{\psi}:\mathbb R^d\to\mathbb R$ by setting
\[\varphi_{k}^{\psi}(x):=\min\left\{\inf_{y\in\mathbb R^d}\left\{\psi(y)+k|x-y|\right\},k\right\}.\]
We recall that $\{\varphi_k^{\psi}\}_{k\in\mathbb N}$ is an increasing sequence of bounded Lipschitz functions bounded from below and pointwise converging to $\psi$.
Hence, by Monotone Convergence Theorem, we have
\[\sup_{\psi\in \Psi}\int_{\mathbb R^d}\psi(x)\,d\mu(x)=\sup_{\psi\in \Psi}\int_{\mathbb R^d}\sup_{k\in\mathbb N}\varphi^\psi_k(x)\,d\mu(x)=
\sup_{k\in\mathbb N}\sup_{\psi\in \Psi}\int_{\mathbb R^d}\varphi^\psi_k(x)\,d\mu(x)=\sup_{\varphi\in \Phi}\int_{\mathbb R^d}\varphi(x)\,d\mu(x),\]
where $\Phi=\{\varphi^\psi_k:\,k\in\mathbb N\setminus\{0\},\,\psi\in \Psi\}$.
Replacing $\Phi$ with its convex hull does not change anything due to the linearity of the integral operator.
\item Since convergence in $W_p(\cdot,\cdot)$ implies $w^*$-convergence, if $\tilde S^{\Phi}$ is a generalized target, then $\tilde S^\Phi_p$ is closed and convex in $\mathscr P_p(\mathbb R^d)$ endowed with the $p$-Wasserstein metric $W_p(\cdot,\cdot)$.
\item We notice that if there exists $\bar x\in \mathbb R^d$ such that $\varphi(\bar x)\le 0$ for all $\varphi\in \Phi$ then the set $\tilde S$ given by \eqref{eq:gentar} is nonempty, since $\delta_{\bar x}\in\tilde S$. 
\end{itemize}
\end{remark}

The last condition of Remark \ref{rem:smoothphi} is indeed not necessary to have the nontriviality of $\tilde S$.

\begin{example}\label{ex:nontrivial}
For every $y\in\mathbb R$, $\varepsilon>0$, define 
\[\varphi^\varepsilon_y(x)=\begin{cases}
-(x+y)^2+\varepsilon,\textrm{ if }|x+y|\le 1,\\ -1+\varepsilon,\textrm{ if }|x+y|\ge 1.
\end{cases}\]
and set $\Phi_{\varepsilon}:=\{\varphi_{y}^\varepsilon:\,y\in\mathbb R\}$.
Clearly, we have that $\varphi^\varepsilon_y$ attains its maximum at $x=-y$ and the value of the maximum is $\varepsilon>0$. Thus the sufficient condition 
of the last assertion in Remark \ref{rem:smoothphi} is violated. For $0<\varepsilon\le \dfrac{1}{12}$ sufficiently small we have 
\[\int_{-1/2}^{1/2}\varphi^\varepsilon_y(x)\,dx\le \int_{-1/2}^{1/2}\varphi^\varepsilon_0(x)\,dx=\varepsilon-\dfrac{1}{12}\le 0,\]
thus the measure $\chi_{[-1/2,1/2]}\mathscr L^1\in \tilde S$. 
\par\medskip\par
Indeed, by the translation invariance of the problem, we have that $\mu_a:=\chi_{[a,a+1]}\mathscr L^1\in \tilde S$ for all $a\in\mathbb R$,
in particular, we have that $\tilde S$ is not tight, hence not $w^*$-compact, since for any $K\subseteq\mathbb R$ it is possible to find $a\in\mathbb R$
such that $\mu_a(\mathbb R\setminus K)=1$.
\end{example}

\begin{lemma}[Compactness]
Let $\tilde S$ be a nonempty generalized target generated by the family $\Phi\subseteq C^0(\mathbb R^d)$.
If there exists $\bar\phi\in\Phi$, $A,C>0$ and $p\ge 1$ such that $\bar \phi(x)\ge A|x|^p-C$, 
then $\tilde S^{\Phi}=\tilde S_p^{\Phi}$ is compact in the $w^*$-topology and in the $W_p$-topology.
\end{lemma}
\begin{proof}
Trivially we have that $\tilde S_p^{\Phi}\subseteq\tilde S^{\Phi}$ for any $p\ge 1$. Conversely, given $\mu\in \tilde S^{\Phi}$, we have 
\[A\cdot\mathrm{m}_p(\mu)-C\le\int_{\mathbb R^d}\bar\phi(x)\,d\mu\le 0,\]
hence $\mu\in\tilde S_p^{\Phi}$ and all the measures in $\tilde S_p^{\Phi}=\tilde S^{\Phi}$ have $p$-moments uniformly bounded by $C/A$.
This means that the $w^*$-topology and $W_p$-topology coincide on $\tilde S^{\Phi}=\tilde S_p^{\Phi}$, which turns out to be tight, according to \cite[Remark 5.1.5]{AGS}, 
and $w^*$-closed, hence $w^*$-compact and $W_p$-compact.
\end{proof}

We mention the following example, which may be relevant for the applications.

\begin{example}
Given a nonempty and closed set $S\subseteq\mathbb R^d$ and $\alpha\ge 0$, a natural choice for $\Phi$ can be for example $\Phi_\alpha=\{d_S(\cdot)-\alpha\}$. 
If $\alpha=0$ we have that $\tilde S^{\Phi_0}=\{\mu\in\mathscr P(\mathbb R^d):\, \mu(\mathbb R^d\setminus S)=0\}$.
More generally, for all $r>0$ let $B_r(S):=\{z\in\mathbb R^d:\,d_S(z)\le r\}$.
Then, if $\mu\in\tilde S^{\Phi_\alpha}$,
\[r\mu(\mathbb R^d\setminus B_r(S))=\int_{\mathbb R^d\setminus B_r(S)}r\,d\mu\le \int_{\mathbb R^d\setminus B_r(S)}d_S(x)\,d\mu(x)\le \alpha,\]
thus, in particular, we must have $\mu(\mathbb R^d\setminus B_r(S))\le \min\left\{1,\dfrac{\alpha}{r}\right\}$ for all $r>0$,
which, if $\alpha$ is sufficiently small can be interpreted as a relaxed version of the case $\alpha=0$.
\end{example}

Given a generalized target $\tilde S\subseteq\mathscr P(\mathbb R^d)$, a natural question is wheter it is possible to \emph{localize} it, i.e., to describe it as the set of all the measures 
supported a certain (closed) subset of $\mathbb R^d$. Equivalently, we want to find a nonempty closed set $S\subseteq\mathbb R^d$, such that, set $\Phi=\{d_S(\cdot)\}$, we have $\tilde S=\tilde S^{\Phi}$.
To this aim, we give the following definition.

\begin{definition}[Classical counterpart of generalized target]\label{def:cctar}
Let $\tilde S\subseteq\mathscr P(\mathbb R^d)$ be a generalized target. Given a set $S\subseteq\mathbb R^d$, 
we say that $S$ is a \emph{classical counterpart of the generalized target} $\tilde S$
if 
\[\tilde S=\{\mu\in\mathscr P(\mathbb R^d):\,\mathrm{supp}\,\mu\subseteq S\}.\]
An analogous definition is given for the classical counterpart of $\tilde S\cap \mathscr P_p(\mathbb R^d)$, $p\ge 1$ by taking intersection of the right hand side with $\mathscr P_p(\mathbb R^d)$.
\end{definition}

\begin{remark}
\begin{itemize}
\item[]
\item From the very definition of classical counterpart, if $\tilde S$ admits $S$ and $S'$ as classical counterparts, then $S=S'$.
\item In general a classical counterpart may not exists: in $\mathbb R$, take $\Phi=\{\phi\}$ where $\phi:\mathbb R\to \mathbb R$, $\phi(y):=|y|-1$.
Defined $\mu_0:=\dfrac12\left(\delta_{0}+\delta_{2}\right)$, we have $\mu_0\in \tilde S_p^{\Phi}$ for every $p\ge 1$. If a classical counterpart $S$ of $\tilde S^{\Phi}$
would exists, by definition it should contain the support of $\mu_0$, i.e. $0,2\in S$. However, $\delta_2\notin \tilde S^{\Phi}$ even if $\mathrm{supp}(\delta_2)\subseteq S$.
So neither $\tilde S^{\Phi}$ nor $\tilde S_p^{\Phi}$ admit a classical counterpart.
\item If $S$ is the classical counterpart of $\tilde S^{\Phi}$ (or $\tilde S_p^{\Phi}$), there exists a representation of $\tilde S^{\Phi}$ as $\tilde S^{\hat \Phi}$, where 
$\hat\Phi=\{\hat\phi\}$ and $\hat\phi(x)\ge 0$ for every $x\in\mathbb R^d$ where the inequality is strict at every $x\notin S$.
In particular we can take $\hat\Phi=\{\arctan\circ d_S\}$ (resp. $\hat\Phi=\{d_S\}$), i.e., we can replace $\Phi$ with the set $\{\arctan\circ d_S\}$ (resp. $\{d_S\}$).
\end{itemize}
\end{remark}

Our aim is now to characterize the generalized target possessing a classical counterpart.

\begin{proposition}[Existence, uniqueness and properties of the classical counterpart]\label{prop:propcctar}
Let $\tilde S\subseteq\mathscr P(\mathbb R^d)$ be a generalized target, $S\subseteq\mathbb R^d$.
\begin{enumerate}
\item if $\tilde S$ admits $S$ as classical counterpart then $S$ is closed;
\item $\tilde S$ admits $S$ as classical counterpart if and only if 
\[\int_{\mathbb R^d}\left[\varphi(x)-\sup_{y\in S}\varphi(y)\right]\,d\mu(x)\le 0,\]
for all $\varphi\in C^0_b(\mathbb R^d)$ and $\mu\in \tilde S$;
\item if $\tilde S$ admits $S$ as classical counterpart, then $\tilde S_p$ admits $S$ as classical counterpart for all $p\ge 1$.
\item If $\tilde S=\tilde S^{\Phi}$ (resp. $\tilde S\cap \mathscr P_p(\mathbb R^d)=\tilde S_p^{\Phi}$), for a suitable $\Phi\subseteq C^0_b(\mathbb R^d)$, admits a classical counterpart $S$, then
\[S=\bigcap_{\phi\in \Phi}\{x\in \mathbb R^d:\, \phi(x)\le 0\}.\]
\end{enumerate}
\end{proposition}
\begin{proof}
\begin{enumerate}
\item[]
\item Assume that $\tilde S$ admits $S$ as a classical counterpart and $\tilde S=\tilde S^\Phi$ for a suitable $\Phi\in C^0_b(\mathbb R^d)$.
In particular, we have $\delta_x\in \tilde S$ for all $x\in S$, i.e. $\phi(x)\le 0$ for all $x\in S$.
Let $\{x_n\}_{n\in\mathbb N}$ be a sequence in $S$ converging to $x\in\mathbb R^d$. Then for all $\varphi\in\Phi$ we have $\varphi(x_n)\le 0$ for all $n\in\mathbb N$, 
which implies $\varphi(x)\le 0$, and so $\delta_x\in \tilde S$. Since $S$ is a classical counterpart of $\tilde S$ and $\mathrm{supp}\delta_x=\{x\}$, we have that thus $x\in S$, so $S$ is closed.
\item $\tilde S$ admits $S$ as classical counterpart if and only if $\tilde S=\overline{\mathrm{co}}\{\delta_x:\,x\in S\}$, where the closure is the weak$^*$ closure in $\mathscr P(\mathbb R^d)$.
Indeed, every measure supported in $S$ is $w^*$-limit of convex combinations of Dirac deltas concentrated in points of $S$, and conversely
all such deltas belong to $\tilde S$ by definition of classical counterpart, and $\tilde S$ is convex and $w^*$-closed.
Recalling formula (5.1.7) in \cite[Remark 5.1.2]{AGS}, we have that $\mu\in\tilde S$  if and only if
\[\int_{\mathbb R^d} \varphi(x)\,d\mu(x)\le \sup_{y\in S}\varphi(y),\]
for all $\varphi\in C^0_b$, as desired.
\item It is sufficient to use the same argument as in (2) but taking the intersection with $\mathscr P_p(\mathbb R^d)$ and the closure w.r.t. $W_p$ distance.
\item Trivially, if there exist $\bar x\in\mathbb R^d$ and $\varphi\in \Phi$ such that $\varphi(\bar x)>0$, then $\delta_{\bar x}\notin \tilde S$, thus $\bar x$ does not belong to the classical
counterpart of $\tilde S$. Conversely, if $\varphi(\bar x)\le 0$ for all $\varphi\in \Phi$, then $\delta_{\bar x}\in \tilde S$, and so $\bar x\in S$ by definition of classical counterpart.
\end{enumerate}
\end{proof}

A useful sufficient condition can be expressed as follows.

\begin{corollary}\label{cor:propcctar}
Assume that for every $\phi\in\Phi$ we have either $\phi(x)\geq 0$ or $\phi(x)\leq 0$ for all $x\in\mathbb R^d$.
Then $\tilde S^{\Phi}$ (and so $\tilde S_p^{\Phi}$) admits classical counterpart.
\end{corollary}
\begin{proof}
Denote by 
\[S=\bigcap_{\phi\in \Phi}\{x\in \mathbb R^d:\, \phi(x)\le 0\}.\]
If for all $\phi\in\Phi$ and $x\in\mathbb R^d$ we had $\phi(x)\le 0$, then we would trivially have $S=\mathbb R^d$ and $\tilde S^\Phi=\mathscr P(\mathbb R^d)$ as desired since
$\delta_x\in \tilde S^{\Phi}$ for all $x\in\mathbb R^d$, thus concluding with the thesis.

Otherwise, let $\mu\in\tilde S^{\Phi}$ and suppose by contradiction that $\mu(\mathbb R^d\setminus S)>0$. Thus there exists $y\in \mathbb R^d\setminus S$ of density $1$ w.r.t. $\mu$.
In particular, there exists a neighborhood $A_y$ of $y$ contained in $\mathbb R^d\setminus S$ such that $\mu(A_y)>0$.
If for all $\varphi\in \Phi$ we had $\varphi(y)\le 0$, we would have $y\in S$, contradicting the fact that $y\notin S$. 
So, according to the assumptions, there exists $\hat\phi\in\Phi$ such that $\hat\phi(x)\ge 0$ for all $x\in\mathbb R^d$ and such that $\hat \phi(y)>0$. Thus we have
\[\sup_{\phi\in\Phi}\int_{\mathbb R^d} \phi(x)\,d\mu(x)\ge \int_{\mathbb R^d} \hat \phi(x)\,d\mu(x)\ge \int_{A_y} \hat\phi(x)\,d\mu(x)>0,\]
hence $\mu\notin \tilde S^\Phi$, leading to a contradiction. Thus $\tilde S^{\Phi}\subseteq\{\mu\in\mathscr P(\mathbb R^d):\, \mathrm{supp}\,\mu\subseteq S\}$.
Since the converse inclusion is always true, equality holds.
\end{proof}

\begin{remark}
The condition of Corollary \ref{cor:propcctar} is not necessary in general.
In $\mathbb R$, take $\Phi=\{\phi_1,\phi_2,\phi_3\}$ where $\phi_i:\mathbb R\to \mathbb R$, $i=1,2,3$ are 
defined to be $\phi_1(x)=\min\{\max\{x,0\},1\}$, $\phi_2(x)=\min\{\max\{-x,-1\},0\}$, $\phi_3(x)=\min\{\max\{x,-1\},1\}$.
Then both $\tilde S_p^{\Phi}$ and $\tilde S^{\Phi}$ admits $S$ as their classical counterpart, with $S=]-\infty,0]$, 
but $\phi_3$ changes its sign.
\end{remark}

We are now ready to state some comparison results between the generalized distance and the classical one. 

\begin{proposition}[Comparison with classical distance]\label{prop:compclasdist}
Let $p\ge 1$, $\mu_0\in \mathscr P_p(\mathbb R^d)$, $\Phi\subseteq C^0_b(\mathbb R^d;\mathbb R)$ be such that $\tilde S_p^{\Phi}\ne\emptyset$, and
define
\begin{equation}\label{eq:clco}S:=\bigcap_{\phi\in \Phi}\{x\in\mathbb R^d:\,\phi(x)\le 0\}.\end{equation}
Then $\tilde d_{\tilde S_p^\Phi}(\mu_0)\le \|d_S\|_{L^p_{\mu_0}}$, and equality holds if and only if the generalized target $\tilde S^\Phi_p$ admits classical
counterpart. In this last case, the classical counterpart of $\tilde S^\Phi_p$ is $S$, moreover $\tilde d_{\tilde S_p^\Phi}^p:\mathscr P_p(\mathbb R^d)\to [0,+\infty[$ is convex.
\end{proposition}
\begin{proof}
If $S=\emptyset$ we have $d_S(x)\equiv +\infty$ at all $x\in\mathbb R^d$ so the statement is trivially true, thus suppose $S\ne\emptyset$.
Since $S$ is closed and nonempty, \cite[Corollary 8.2.13]{AuF} implies
the existence of a Borel map $g:\mathbb R^d\to S$ such that $|x-g(x)|=d_S(x)$.
We have
\begin{equation*}
\mathrm{m}^{1/p}_p(g\sharp\mu_0)=\|g\|_{L^p_{\mu_0}}\le
\|\mathrm{Id}_{\mathbb R^d}-g\|_{L^p_{\mu_0}}+\|\mathrm{Id}_{\mathbb
R^d}\|_{L^p_{\mu_0}}\le \|d_S\|_{L^p}+\mathrm{m}_p^{1/p}(\mu_0)<+\infty,
\end{equation*}
moreover, for all $\phi\in\Phi$, we have
\begin{equation*}
\int_{\mathbb R^d}\phi(x)\,dg\sharp\mu_0(x)=\int_{\mathbb
R^d}\phi(g(y))\,d\mu_0(y)\le 0,
\end{equation*}
since $g(y)\in S$ for all $y\in\mathbb R^d$ and so $\phi\circ g(y)\le 0$
for all $y\in\mathbb R^d$.
Therefore, $g\sharp\mu_0\in\tilde S^\Phi_p$, and so
\[\tilde d^p_{\tilde S^\Phi}(\mu_0)\le W^p_p(\mu_0,g\sharp\mu_0)\le
\|\mathrm{Id}_{\mathbb R^d}-g\|^p_{L^p_{\mu_0}}=\|d_S\|^p_{L^p_{\mu_0}}.\]
\par\medskip\par
Assume now that $\tilde S^\Phi_p$ admits classical counterpart. As noticed in Proposition \ref{prop:propcctar}, $S$ must be the classical counterpart of $\tilde S^\Phi_p$.
For every $\nu_0\in \tilde S^{\Phi}_p$ we have thus $\mathrm{supp}\,\nu_0\subseteq S$ and hence $|x-y|\ge d_S(x)$ for all $\pi\in \Pi(\mu_0,\nu_0)$ and $\pi$-a.e. $(x,y)\in\mathbb R^d\times\mathbb R^d$.
This leads to
\[\iint_{\mathbb R^d\times\mathbb R^d}|x-y|^p\,d\pi(x,y)\ge \int_{\mathbb R^d}d_S^p(x)\,d\mu_0(x).\]
By taking the infimum on $\pi\in \Pi(\mu_0,\nu_0)$ and then on $\nu_0\in \tilde S^{\Phi}_p$, we obtain $\tilde d_{\tilde S_p^\Phi}(\mu_0)\ge \|d_S\|_{L^p_{\mu_0}}$,
thus equality holds.
\par\medskip\par
Without the assumption of existence of a classical counterpart for $\tilde S^{\Phi}_p$, 
the inequality $\tilde d_{\tilde S_p^\Phi}(\mu_0)\le \|d_S\|_{L^p_{\mu_0}}$ is strict.
Indeed, since $\tilde S^{\Phi}_p$ does not admit $S$ as a classical counterpart, there exist a measure $\mu\in\tilde S_p^\Phi$ and $n\in\mathbb N$ such that 
\[\mu\left(\left\{z\in\mathbb R^d:\, d_S(z)>\dfrac{1}{n}\right\}\right)>0,\]
and so there exists a Borel set $A\subseteq\mathbb R^d$ and $\varepsilon>0$ such that $d^p_S(z)\ge \varepsilon$ for $\mu$-a.e. $z\in A$, $\mu(A)>0$.
This implies
\[0=\tilde d^p_{\tilde S^{\Phi}_p}(\mu)<\varepsilon\mu(A)\le \int_A\,d^p_S(z)\,d\mu(z)\le \int_{\mathbb R^d}d^p_S(x)\,d\mu(x).\]
Finally, the last statement is trivial, and it follows from the fact that 
\[\tilde d^p_{\tilde S_p^\Phi}(\mu)=\int_{\mathbb R^d}d_S^p(x)\,d\mu,\] 
is linear in $\mu$.
\end{proof}

Without the $p$-th power, the generalized distance in the case of the Proposition \ref{prop:compclasdist} above may fail to be convex.

\begin{example}
Let $p>1$. In $\mathbb R^2$, consider $P=(0,0)$, $Q_1=(1,0)$, $Q_2=\left(0,2^{1/p}\right)$. Set $S=\{P\}$, $\Phi=\{d_S(\cdot)\}$, hence $\tilde S_p^\Phi:=\left\{\delta_P\right\}$,
and define $\nu_{\lambda}=\lambda\delta_{Q_1}+(1-\lambda)\delta_{Q_2}$, $\lambda\in[0,1]$.
By Proposition \ref{prop:compclasdist}, we have
\[\tilde d^p_{\tilde S_p^\Phi}(\nu_{\lambda})=W_p^p(\delta_P,\nu_{\lambda})=\lambda+2(1-\lambda)=2-\lambda,\]
whence $\tilde d_{\tilde S_p^\Phi}(\nu_{\lambda})=\sqrt[p]{2-\lambda}$, which is not convex.
\end{example}

In the metric space $\mathscr P_p(\mathbb R^d)$ endowed with the $W_p$-distance, another concept of convexity can be given,
related more to the metric structure rather than to the linear one inherited by the set of all Borel signed measures.

\medskip

Given any product space $X^N$ ($N\geq 1$), in the following we denote with $\mathrm{pr}^i\colon X^N\to X$ the projection on the $i$--th component, i.e., $\mathrm{pr}^i(x_1,\ldots, x_N)=x_i$.

\begin{definition}[Geodesics]
Given a curve $\boldsymbol\mu=\{\mu_t\}_{t\in [0,1]}\subseteq\mathscr P_p(\mathbb R^d)$, we say that it is a \emph{(constant speed) geodesic} if for all $0\leq s\leq t\leq 1$
we have
\[W_p(\mu_s,\mu_t)=(t-s)W_p(\mu_0,\mu_1).\]
In this case, we will also say that the curve $\boldsymbol\mu$ is a \emph{geodesic connecting $\mu_0$ and $\mu_1$}.
\end{definition}

\begin{theorem}[Characterization of geodesics] 
Let $\mu_0,\mu_1\in\mathscr P_p(\mathbb R^d)$ and let $\pi\in\Pi_o^p(\mu_0,\mu_1)$ be an optimal 
transport plan between $\mu_0$ and $\mu_1$, i.e.
\[W_p^p(\mu_0,\mu_1)=\iint_{\mathbb R^d\times\mathbb R^d} |x_1-x_2|^p\,d\pi(x_1,x_2)\,.\]
Then the curve $\boldsymbol\mu=\{\mu_t\}_{t\in [0,1]}$ defined by
\begin{equation}\label{eq:charact_geodesic}
\mu_t:=\big((1-t)\,\mathrm{pr}^1+t\,\mathrm{pr}^2\big)\sharp \pi~\in~\mathscr P_p(\mathbb R^d),
\end{equation}
is a (constant speed) geodesic connecting $\mu_0$ and $\mu_1$.

Conversely, any (constant speed) geodesic  $\boldsymbol\mu=\{\mu_t\}_{t\in [0,1]}$ connecting $\mu_0$ and $\mu_1$ admits 
the representation~\eqref{eq:charact_geodesic} for a suitable plan $\pi\in\Pi_o^p(\mu_0,\mu_1)$.
\end{theorem}
\begin{proof}
See \cite[Theorem 7.2.2]{AGS}. 
\end{proof}

\begin{definition}[Geodesically and strongly geodesically convex sets]
A subset $A\subseteq \mathscr P_p(\mathbb R^d)$ is said to be 
\begin{enumerate}
\item \emph{geodesically convex} if for every pair of measures $\mu_0,\mu_1$ in $A$, there exists a geodesic connecting $\mu_0$ and $\mu_1$ which is contained in $A$.
\item \emph{strongly geodesically convex} 
if for every pair of measures $\mu_0,\mu_1$ in $A$ and for every admissible transport plan $\pi\in\Pi(\mu_0,\mu_1)$, the curve $t\mapsto\mu_t$ defined by~\eqref{eq:charact_geodesic}
is contained in $A$.
\end{enumerate}
\end{definition}

The interest in this alternative concept of convexity comes from the fact that, in many problems, functionals defined on probability measures are convex along geodesics 
(a notion related to geodesically convex sets) and not convex with respect to the linear structure in the usual sense. We refer to \cite[Section 9.1]{AGS}
for further details.

\begin{remark} 
Notice that, even if the notations do not highlight this fact, the notions of \emph{geodesic} and \emph{geodesical convexity} depend on the exponent $p$ 
which has been fixed.
\end{remark}

\begin{proposition} [Strong geodesic convexity of $\tilde S_p^\Phi$]
Let $p\ge 1$, $\Phi$ as in Definition \ref{def:gentar} and assume that all the elements of $\Phi$ are also convex.
Then the generalized target $\tilde S_p^{\Phi}$ is strongly geodesically convex.
\end{proposition}
\begin{proof}
Let $\mu_0,\mu_1\in\tilde S_p^\Phi$ and let $\pi\in\Pi(\mu_0,\mu_1)$ be an admissible transport plan between $\mu_0$ and $\mu_1$. 
Consider the corresponding curve $\boldsymbol\mu=\{\mu_t\}_{t\in [0,1]}$ defined by~\eqref{eq:charact_geodesic}, and fix $t\in [0,1]$.
We have for every $\phi(\cdot)\in\Phi$
\begin{align*}
\int_{\mathbb R^d} & \phi(x)\,d\mu_t(x)\le\\
&\leq (1-t)\iint_{\mathbb R^d\times\mathbb R^d} \phi\left(\mathrm{pr}^1(\xi,\eta)\right)\,d\pi(\xi,\eta)
+t\iint_{\mathbb R^d\times\mathbb R^d} \phi\left(\mathrm{pr}^2(\xi,\eta)\right)\,d\pi(\xi,\eta)\\
&=(1-t)\int_{\mathbb R^d} \phi(x)\,d\mu_0(x)
+t\int_{\mathbb R^d} \phi(y)\,d\mu_1(y)\le 0\,,
\end{align*}
since $\mathrm{pr}^i\sharp\pi$ are the marginal measures of $\pi$, which belong to $\tilde S_p^\Phi$.
The conclusion follows from the arbitrariness of $\phi(\cdot)\in \Phi$.
\end{proof}

\begin{remark}
In particular, considering also the first item in Remark \ref{rem:smoothphi}, the above result holds for $\Phi:=\{d_S(\cdot)-\alpha\}$ when $S$ is nonempty, closed and convex, and $\alpha\in[0,1]$.
In this case, since in the above proof we use only the convexity property of $d_S(\cdot)$, the 
statement holds also if we equip $\mathbb R^d$ with a different norm than the Euclidean one.
\end{remark}

We conclude this section by investigating the semiconcavity properties of the generalized distance along geodesics. 
The case $p=2$ is particularly easy thanks to the geometric structure of the metric space $\mathscr P_2(\mathbb R^d)$.

\begin{proposition}[Semiconcavity of $\tilde d^{\,2}_{\tilde S_2^\Phi}$]\label{prop:semiconc_2}
Let $\tilde S_2^\Phi$ be a generalized target in $\mathscr P_2(\mathbb R^d)$.
Then the square of the generalized distance satisfies the following \emph{global semiconcavity inequality:} for every $\mu_0,\mu_1\in \mathscr P_2(\mathbb R^d)$ and every $t\in[0,1]$
\[\tilde d^{\,2}_{\tilde S_2^\Phi}(\mu_t)\ge (1-t)\,\tilde d^{\,2}_{\tilde S_2^\Phi}(\mu_0)+t\,\tilde d^{\,2}_{\tilde S_2^\Phi}(\mu_1)-t(1-t)\,W^2_2(\mu_0,\mu_1),\]
where $\boldsymbol \mu=\{\mu_t\}_{t\in[0,1]}$ is any constant speed geodesic for $W_2$ joining $\mu_0$ and $\mu_1$.
\end{proposition}
\begin{proof}
Owing to \cite[Theorem 7.3.2]{AGS}, we have that for any measure $\sigma\in\mathscr P_2(\mathbb R^d)$ the function
$\mu\mapsto W_2^2(\mu,\sigma)$ is semiconcave along geodesics, with semiconcavity constant independent by $\sigma$, i.e. it satisfies for every $t\in[0,1]$
\[W_2^2(\mu_t,\sigma)+t(1-t)\,W^2_2(\mu_0,\mu_1)\ge (1-t)\,W_2^2(\mu_0,\sigma)+t\,W_2^2(\mu_1,\sigma).\]
The conclusion follows by passing to the infimum on $\sigma\in \tilde S_2^\Phi$.
\end{proof}

In the case $p\neq 2$ we need additional requirements on $\Phi$.

\begin{proposition}[Semiconcavity of $\tilde d^{\,p}_{\tilde S_p^\Phi}$]\label{prop:semicon}
Let $p\ge 1$, and $\tilde S_p^\Phi$ be a generalized target. Assume
that $\tilde S_p^\Phi$ admits a classical counterpart $S\subseteq\mathbb R^d$.
Let $K\subseteq\mathbb R^d\setminus S$ be compact and convex.
Then the $p$-th power of the generalized distance $\tilde d_{\tilde S_p^\Phi}(\cdot)$ 
satisfies the following \emph{local semiconcavity inequality:} 
there exists a constant $C=C(p,K)>0$ such that for every $\mu_0,\mu_1\in\mathscr P_p(K)$ we have 
\begin{equation}\label{eq:gendis_semiconc}
\tilde d^{\,p}_{\tilde S_p^\Phi}(\mu_t)\ge (1-t)\,\tilde d^{\,p}_{\tilde S_p^\Phi}(\mu_0)+t\,\tilde d^{\,p}_{\tilde S_p^\Phi}(\mu_1)-Ct(1-t)\,W^{\min\{p,2\}}_p(\mu_0,\mu_1),
\end{equation}
where $\boldsymbol \mu=\{\mu_t\}_{t\in[0,1]}$ is any constant speed geodesic for $W_p$ joining $\mu_0$ and $\mu_1$.
\end{proposition}
\begin{proof}
In this proof to make clearer the notation we will omit the superscript $\Phi$, since $\Phi$ is fixed.
Under the above assumptions, and recalling Proposition \ref{prop:compclasdist}, we have $\tilde d_{\tilde S_p}(\mu_0)=\|d_S\|_{L^p_{\mu_0}}$.\par\medskip\par
Given $x_0,x_1\in K$ and $t\in [0,1]$ we set 
\[x_t:=(1-t)x_0+tx_1,\qquad\qquad d_t:=(1-t)d_{S}(x_0)+t d_{S}(x_1).\]
%
%
According to \cite[Proposition 2.2.2]{CSb}, there exists $c=c(K)>0$ such that
$d_S$ satisfies the following inequality for all $x_0,x_1\in K$:
\[d_S(x_t)\ge d_t-ct(1-t)|x_0-x_1|^2,\]
%
%
By using \cite[Proposition 2.1.12 (i)]{CSb}, we obtain that 
\begin{equation}\label{eq:semiconc_dp}
d_S^p(x_t)\ge (1-t)\,d^p_S(x_0)+t\, d^p_S(x_1) -C' t(1-t)|x_0-x_1|^{\min\{p,2\}},
\end{equation}
with $C'=C'(p,K)$.\par\medskip\par
For any Borel sets $A,B\subseteq\mathbb R^d$ and $\pi\in \Pi(\mu_0,\mu_1)$, we now have $\mathrm{supp}(\pi)\subseteq K\times K$. 
Therefore, we choose a transport plan $\pi\in \Pi_o^p(\mu_0,\mu_1)$ realizing the $p$-Wasserstein distance between $\mu_0$ and $\mu_1$, so that the representation in formula \eqref{eq:charact_geodesic} holds, and we integrate the estimate~\eqref{eq:semiconc_dp} to find that
\begin{align*}
\int_{\mathbb R^d}d_S^p(x)\,d\mu_t=\iint_{\mathbb R^d\times \mathbb R^d}d_S^p(x_t)\,d\pi 
&\geq (1-t)\int_{\mathbb R^d}d^p_S(x)\,d\mu_0+t\int_{\mathbb R^d}d^p_S(x)\,d\mu_1\\
&\hspace{.7cm}-C'\,t\,(1-t)\,\iint_{\mathbb R^d\times \mathbb R^d}|x_0-x_1|^{\min\{p,2\}}\,d\pi,
\end{align*}
where $\boldsymbol \mu=\{\mu_t\}_{t\in[0,1]}\subseteq\mathscr P_p(\mathbb R^d)$ is the constant speed geodesic corresponding to $\pi$. But 
according to Proposition \ref{prop:compclasdist}, there holds
\[\tilde d^{\,p}_{\tilde S_p}(\mu_t)=\int_{\mathbb R^d}d_S^p(x)\,d\mu_t(x),
\qquad\mbox{and}\qquad
\tilde d^{\,p}_{\tilde S_p}(\mu_i)=\int_{\mathbb R^d}d_S^p(x)\,d\mu_i(x),\quad i=0,1,\]
and applying Jensen's inequality to the concave map $\xi\mapsto \xi^{\gamma/p}$ on $\mathbb R^+$, with $\gamma=\min\{p,2\}$, we obtain that
\[\iint_{\mathbb R^d\times \mathbb R^d}|x_0-x_1|^{\min\{p,2\}}\,d\pi
\leq \begin{cases} 
\displaystyle \iint_{\mathbb R^d\times\mathbb R^d}|x_0-x_1|^p\,d\pi,&\textrm{ for }1\le p<2,\\ &\\ 
\left(\displaystyle\iint_{\mathbb R^d\times\mathbb R^d}|x_0-x_1|^p\,d\pi\right)^{2/p},&\textrm{ for }p\ge 2.
\end{cases}\]
We thus conclude that
\[\tilde d^{\,p}_{\tilde S_p}(\mu_t)\ge (1-t)\,\tilde d^{\,p}_{\tilde S_p}(\mu_0)+t\, \tilde d^{\,p}_{\tilde S_p}(\mu_1)- C'\,t\,(1-t)\, W_p^{\min\{p,2\}}(\mu_0,\mu_1),\]
and the proof is completed.  
\end{proof}
 
\begin{remark} 
Notice that inequality~\eqref{eq:gendis_semiconc} implies that, for $p\ge 2$ and under the assumption of Proposition~\ref{prop:semicon}, 
the functional $-\tilde d^{\,p}_{\tilde S_p}(\cdot)\colon\mathscr P_p(K)\to \,]-\infty,0]$ is  $\lambda$-geodesically convex, 
in the sense of \cite[Definition 9.1.1]{AGS}, with $\lambda=-2C$.
\end{remark}

\section{The generalized minimum time function}\label{sec:mintime}

\begin{definition}[Generalized minimum time function]
Given a generalized target $\tilde S_p=\tilde S_p^\Phi$ defined in Definition \ref{def:gentar}, we define the generalized minimum time function $\tilde T_p:\mathscr P_p(\mathbb R^d)\to [0,+\infty]$ by 
\[\tilde T_p(\mu):=\inf\{T\ge 0:\, \textrm{ there exists }\boldsymbol\mu=\{\mu_t\}_{t\in[0,T]}\in\mathcal A^p_{[0,T]}(\mu)\textrm{ s.t. }\mu_T\in\tilde S_p\},\]
where we set $\inf\emptyset=+\infty$ by convention.
We say that $\boldsymbol\mu=\{\mu_t\}_{t\in[0,T]}\in\mathcal A^p_{[0,T]}(\mu)$ is time optimal from $\mu$ if $\tilde T_p(\mu)\le T<+\infty$ and $\mu_{\tilde T_p(\mu)}\in\tilde S_p$. 
\end{definition}

\begin{proposition}[Properties of $\tilde T_p$]\label{prop:DPP}
Assume Hypothesis \ref{HP} for $F$. Then
\begin{enumerate}
\item for any $\mu\in\mathscr P_p(\mathbb R^d)$ with $\tilde T_p(\mu)<+\infty$ there exists a time optimal admissible trajectory from $\mu$;
\item the function $\tilde T_p(\cdot)$ is lower semicontinuous; 
\item the following Dynamic Programming Principle holds
\begin{equation}\label{eq:DPP}\tilde T_p(\mu)=\inf\{t+\tilde T_p(\mu_t):\, \boldsymbol\mu=\{\mu_t\}_{t\in[0,T]}\in \mathcal A^p_{[0,T]}(\mu),\, T>0\}.\end{equation}
In particular, $t\mapsto t+\tilde T_p(\mu_t)$ is nondecreasing along every admissible trajectory $\boldsymbol\mu=\{\mu_t\}_{t\in[0,T]}\in \mathcal A^p_{[0,T]}(\mu)$, 
and it is constant if and only if $\boldsymbol\mu=\{\mu_t\}_{t\in[0,T]}\in \mathcal A^p_{[0,T]}(\mu)$ is the restriction to $[0,T]\cap [0,\tilde T_p(\mu)]$ of an optimal trajectory.
\end{enumerate}
\end{proposition}
\begin{proof}
\begin{enumerate}\item[]
\item Fix $\mu\in\mathscr P_p(\mathbb R^d)$ with $\tilde T_p(\mu)<+\infty$. For any $n\in\mathbb N\setminus\{0\}$ there exists $T_n>0$ and $\boldsymbol{\mu^{(n)}}=\{\mu^{(n)}_t\}_{t\in[0,T_n]}\in \mathcal A^p_{[0,T_n]}(\mu)$ such that $\mu^{(n)}_{T_n}\in \tilde S_p$ and $T_n\le \tilde T_p(\mu)+1/n$.
We can extend each $\boldsymbol{\mu^{(n)}}$ to an admissible curve defined on $\tilde T_p(\mu)+1$ (possibly concatenating it with an element of $\mathcal A_{[T_n,\tilde T_p(\mu)+1]}(\mu^{(n)}_{T_n})$, which is nonempty for all $n\in\mathbb N\setminus\{0\}$). Thus, without loss of generality, we may assume that we have a sequence $\boldsymbol{\hat\mu^{(n)}}=\{\hat\mu^{(n)}_t\}_{t\in[0,T]}\in\mathcal A^p_{[0,T]}(\mu)$ of admissible trajectories, which are all defined in $[0,T]$ with $T=\tilde T_p(\mu)+1$, and satisfying $\hat\mu^{(n)}_{T_n}\in \tilde S_p$ where $T_n\le \tilde T_p(\mu)+1/n$.
Recalling the compactness of $\mathcal A^p_{[0,T]}(\mu)$ (see Corollary \ref{cor:exist}), up to passing to a subsequence, the sequence of curves $\{\boldsymbol{\hat\mu^{(n)}}\}_{n\in\mathbb N}$ uniformly converges to $\boldsymbol{\hat\mu^{\infty}}=\{\hat\mu^{\infty}_t\}_{t\in [0,T]}$ and, moreover, we have $T_n\to \ell$. In particular, $\hat\mu^{(n)}_{T_n}\to \hat\mu^{\infty}_\ell$ which, by the closedness of $\tilde S_p$, implies $\hat\mu^{\infty}_\ell\in \tilde S_p$, and so $\tilde T_p(\mu)\le \ell$. But passing to the limit in $T_n\le \tilde T_p(\mu)+1/n$ yields the reverse inequality, thus $\ell=\tilde T_p(\mu)$,
hence $\boldsymbol{\hat\mu^{\infty}}$ is optimal.
\item Let $\{\mu^{(n)}\}_{n\in\mathbb N}\subseteq \mathscr P_p(\mathbb R^d)$ be a $W_p$-converging sequence satisfying $\mu^{(n)}\to \mu^{\infty}$
and \\$\displaystyle\liminf_{n\to+\infty}\tilde T_p(\mu^{(n)})=: \ell\in\mathbb R$. If $\ell=+\infty$ there is nothing to prove, so let us assume $\ell<+\infty$.
As before, up to concatenation and restriction and by taking $n$ sufficiently large, this implies that there exists a sequence $\{\boldsymbol{\mu^{(n)}}\}_{n\in\mathbb N}$ such that $\boldsymbol{\mu^{(n)}}=\{\mu^{(n)}_t\}_{t\in[0,\ell+1]}\in\mathcal A_{[0,\ell+1]}(\mu^{(n)})$
and $\mu^{(n)}_{\tilde T_p(\mu^{(n)})}\in\tilde S_p$  for all $n\in\mathbb N$. \par By Theorem \ref{thm:filippov}, there exists a sequence
$\{\boldsymbol{\hat\mu^{(n)}}=\{\hat\mu^{(n)}_t\}_{t\in[0,\ell+1]}\}_{n\in\mathbb N}\subseteq\mathcal A_{[0,\ell+1]}(\mu^{\infty})$ such that
\[W_p(\hat\mu^{(n)}_t,\mu^{(n)}_t)\le D\cdot W_p(\mu^{(n)},\mu^{\infty}),\]
for all $t\in [0,\ell+1]$, where $D:=2^{\frac{p-1}{p}}e^{L(2+Le^{L(\ell+1)})(\ell+1)}$. Recalling the compactness of $\mathcal A^p_{[0,\ell+1]}(\mu^\infty)$ (see Corollary \ref{cor:exist}), up to a passing to a subsequence, the sequence of curves $\{\boldsymbol{\hat\mu^{(n)}}\}_{n\in\mathbb N}$ uniformly converges to $\boldsymbol{\hat\mu^{\infty}}=\{\hat\mu^{\infty}_t\}_{t\in [0,\ell+1]}$, in particular, we have that 
\begin{align*}
W_p\left(\mu^{(n)}_{\tilde T_p(\mu^{(n)})},\hat\mu^{\infty}_{\ell}\right)
\le&W_p\left(\mu^{(n)}_{\tilde T_p(\mu^{(n)})},\hat\mu^{(n)}_{\tilde T_p(\mu^{(n)})}\right)+W_p\left(\hat\mu^{(n)}_{\tilde T_p(\mu^{(n)})},\hat\mu^{\infty}_{\tilde T_p(\mu^{(n)})}\right)+\\&+W_p\left(\hat\mu^{\infty}_{\tilde T_p(\mu^{(n)})},\hat\mu^{\infty}_{\ell}\right)\\
\le&D\cdot W_p\left(\mu^{(n)},\mu^{\infty}\right)+\sup_{t\in[0,\ell+1]}W_p\left(\hat\mu^{(n)}_{t},\hat\mu^{\infty}_{t}\right)+\\&+W_p\left(\hat\mu^{\infty}_{\tilde T_p(\mu^{(n)})},\hat\mu^{\infty}_{\ell}\right).
\end{align*}
By taking the limit for $n\to +\infty$, we have that $W_p\left(\mu^{(n)}_{\tilde T_p(\mu^{(n)})},\hat\mu^{\infty}_{\ell}\right)\to 0$, 
hence, by the closedness of $\tilde S_p$, we obtain $\hat\mu^{\infty}_\ell\in \tilde S_p$, and so $\tilde T_p(\mu^\infty)\le \ell$.
\item Let $\boldsymbol\mu=\{\mu_t\}_{t\in [0,T]}\in\mathcal A^p_{[0,T]}(\mu)$ and 
$\boldsymbol{\hat\mu^{(t)}}\in \mathcal A^p_{[0,\tilde T_p(\mu_t)]}(\mu_t)$ such that $\boldsymbol{\hat\mu^{(t)}}$ is optimal for $\mu_t$ 
(such an optimal trajectory exists by item (1)).
For any $t\in[0,T]$, the concatenation $\boldsymbol\mu_{|[0,t]}\odot\boldsymbol{\hat\mu^{(t)}}\in \mathcal A^p_{[0,t+\tilde T_p(\mu_t)]}(\mu)$, 
and so $\tilde T_p(\mu)\le t+\tilde T_p(\mu_t)$ for every $t\in [0,T]$, $\boldsymbol\mu=\{\mu_t\}_{t\in [0,T]}\in\mathcal A^p_{[0,T]}(\mu)$,
$t>0$, giving the first inequality in \eqref{eq:DPP}.
In particular, for $0\le t\le s\le T$, we have
\[\tilde T_p(\mu)\le t+\tilde T_p(\mu_t)\le t+(s-t)+\tilde T_p(\mu_s)=s+\tilde T_p(\mu_s),\]
since the restriction of $\boldsymbol\mu$ to $[t,T]$ is an admissible trajectory from $\mu_t$. Thus $t\mapsto t+\tilde T_p(\mu_t)$ is 
nondecreasing along all the admissible trajectories.
If $\boldsymbol\mu$ is an optimal trajectory, by taking $s=\tilde T_p(\mu)$ we have $\tilde T_p(\mu_s)=0$ and so $\tilde T_p(\mu)=t+\tilde T_p(\mu_t)$ 
for all $t\in [0,\tilde T_p(\mu)]$, which gives equality in \eqref{eq:DPP}.
Finally, assume that $t\mapsto t+\tilde T_p(\mu_t)$ is constant along an admissible trajectory $\boldsymbol\mu\in \mathcal A^p_{[0,T]}(\mu)$.
By \eqref{eq:DPP} we have that $\tilde T_p(\mu)=t+\tilde T_p(\mu_t)$ for all $t\in [0,T]$. 
If $T\ge \tilde T_p(\mu)$, this implies that $\boldsymbol\mu$ is optimal, since by taking $t=\tilde T_p(\mu)$ we obtain $\tilde T_p(\mu_{\tilde T_p(\mu)})=0$ 
and so $\mu_{\tilde T_p(\mu)}\in \tilde S_p$. If $T<\tilde T_p(\mu)$ we concatenate $\boldsymbol\mu$ with an optimal trajectory 
$\boldsymbol{\hat\mu}=\{\hat\mu_s\}_{s\in[0,\tilde T_p(\mu_T)]}\in\mathcal A_{[0,\tilde T_p(\mu_T)]}(\mu_T)$ for $\mu_T$. 
Set $\boldsymbol\mu\odot\boldsymbol{\hat\mu}=\{\tilde\mu_s\}_{s\in[0,T+\tilde T_p(\mu_T)]}$.
In particular, we have $\tilde T_p(\mu_T)=s+\tilde T_p(\hat\mu_s)$ for all $s\in [0,\tilde T_p(\mu_T)]$, thus $\tilde T_p(\mu)=\tau+\tilde T_p(\tilde\mu_\tau)$ 
for all $\tau\in [0,\tilde T_p(\mu)]$. By taking $\tau=\tilde T_p(\mu)$ we obtain $\tilde T_p(\mu_{\tilde T_p(\mu)})=0$ and so $\mu_{\tilde T_p(\mu)}\in \tilde S_p$ 
thus the concatenation $\boldsymbol\mu\odot\boldsymbol{\hat\mu}$ is an optimal trajectory, whose restriction to $[0,T]$ is $\boldsymbol\mu$.
\end{enumerate}
\end{proof}

The following definition of Small-Time Local Attainability (STLA) has been introduced in \cite{KQ} for finite-dimensional control systems, but can be easily generalized in our framework.

\begin{definition}[STLA for Wasserstein spaces]\label{def:STLA-W}
We say that the system with generalized target $\tilde S_p$ satisfies the STLA property if 
\par\medskip\par
\emph{Property \textbf{\textup{(STLA)}}}: 
for any $\varepsilon>0$ and $\hat\mu\in \tilde S_p$ there exists $\delta>0$
such that $\tilde T_p(\mu)\le \varepsilon$ for any $\mu\in\mathscr P_p(\mathbb R^d)$ satisfying 
$W_p(\mu,\hat\mu)\le \delta$.
\end{definition}

The link between STLA and continuity of the generalized minimum time is provided by the following result.

\begin{proposition}[STLA and continuity of $\tilde T_p$]\label{prop:STLAcontT}
 Let $\tilde S_p$ be 
a generalized target. Assume Hypothesis \ref{HP} for $F$, and that \textbf{\textup{(STLA)}} holds for the system.
Then $\tilde T_p:\mathscr P_p(\mathbb R^d)\to [0,+\infty]$ is continuous at every point where it is finite.
\end{proposition}
\begin{proof}
Recalling the l.s.c. of $\tilde T_p(\cdot)$, given $\mu\in\mathscr P_p(\mathbb R^d)$ with $\tilde T_p(\mu)=+\infty$,
we have $\displaystyle\lim_{n\to+\infty}\tilde T_p(\mu^{(n)})=+\infty$ for every sequence $\{\mu^{(n)}\}_{n\in\mathbb N}$ converging to $\mu$ in $W_p$.
\par\medskip\par
Therefore, we assume $T:=\tilde T_p(\mu)<+\infty$. Since $\tilde T_p(\cdot)$ is l.s.c., it is enough to prove that for all $\{\bar\mu^{(n)}\}_n\subseteq\mathscr P_p(\mathbb R^d)$ such that $W_p(\bar\mu^{(n)},\mu)\to0$ as $n\to+\infty$, we have
\[\limsup_{n\to+\infty}\tilde T_p(\bar\mu^{(n)})\le T.\]
Fix an optimal trajectory $\boldsymbol\mu^{\infty}:=\{\mu^\infty_t\}_{t\in[0,T]}$ starting from $\mu^{\infty}_{|t=0}=\mu$.
Let $\{\mu^{(n)}\}_{n\in\mathbb N}$ be a sequence converging to $\mu$ in $W_p$ and such that $\displaystyle\lim_{n\to+\infty}\tilde T_p(\mu^{(n)})$ exists.
By Theorem \ref{thm:filippov},
there exists a sequence of admissible trajectories $\{\boldsymbol\mu^{(n)}\}_{n\in\mathbb N}$ such that 
\begin{itemize}
\item $\boldsymbol\mu^{(n)}=\{\mu^{(n)}_t\}_{t\in [0,T]}$, $\mu^{(n)}_0=\mu^{(n)}$ for all $n\in\mathbb N$ and
\item $\tilde d_{\tilde S_p}(\mu^{(n)}_T)\le W_p(\mu^{(n)}_T,\mu^{\infty}_T)\le D\cdot W_p(\mu^{(n)},\mu)$, recalling that $\mu^{\infty}_T\in\tilde S_p$,
\end{itemize}
where $D:=2^{\frac{p-1}{p}}e^{L(2+Le^{LT})T}$.
In particular, by \textbf{\textup{(STLA)}}, given $\varepsilon>0$ there exists $n_\varepsilon\in\mathbb N$ such that for all $n>n_\varepsilon$ we have 
$\tilde T_p(\mu^{(n)}_T)\le \varepsilon$. By Dynamic Programming principle, we have 
\[\tilde T_p(\mu^{(n)})\le T+\tilde T_p(\mu^{(n)}_T)\le T+\varepsilon,\]
By letting $n\to +\infty$ and $\varepsilon\to 0$, we have
\[\lim_{n\to+\infty}\tilde T_p(\mu^{(n)})\le T.\]
We conclude by the arbitrariness of the sequence $\{\mu^{(n)}\}_{n\in\mathbb N}$.
\end{proof}

\begin{definition}\label{def:Lsigma}
Given $\Phi\subset C_b^0(\mathbb R^d)$, $\phi\in\Phi$ and $\mu\in \mathscr P(\mathbb R^d)$ we define
\[L_\phi(\mu):=\displaystyle\int_{\mathbb R^d}\phi(x)\,d\mu(x),\hspace{2cm}\sigma_\Phi(\mu):=\sup_{\phi\in\Phi}L_\phi(\mu).\]
\end{definition}

Our aim is to provide a sufficient condition for \textbf{\textup{(STLA)}}, following the line of \cite{LM} and \cite{MR} for finite-dimensional systems. We recall that the l.s.c. of $\tilde T_p(\cdot)$ was already showed in \cite[Theorem 4]{CMNP} in a simplified setting,
while a stronger sufficient condition was provided in \cite[Theorem 4.1]{Cav} to prove the Lipschitz continuity regularity. The continuity of $\tilde T_p(\cdot)$ was a crucial assumption also in \cite[Theorem 8]{CMNP}
to prove that it solves an Hamilton-Jacobi-Bellman equation in Wasserstein space.

\medskip

The following definition establishes a quantitative estimate of the maximal infinitesimal decreasing of the functions $\phi\in\Phi$ defining the generalized target, along the admissible trajectories of the system.

\begin{definition}\label{def:Sattain}
We say that the generalized target $\tilde S^\Phi_p$ is \emph{$(r,Q)$-attainable} if there exist continuous maps
\begin{equation*}
r:[0,+\infty[\to\left[0,\min\left\{1,\frac{1}{2L}\right\}\right],\quad Q:\left[0,\min\left\{1,\frac{1}{2L}\right\}\right]\times[0,+\infty[\to\mathbb R
\end{equation*}
such that
\begin{enumerate}
\item $r(q)=0$ if and only if $q=0$;
\item $Q(r(q),q)<0$ for all $q\in]0,+\infty[$;
\item the function $q\mapsto \dfrac{r(q)}{|Q(r(q),q)|}$ is decreasing and integrable on $[0,+\infty[$.
\item for any $\mu\in\mathscr P_p(\mathbb R^d)\setminus\tilde S^\Phi_p$ there exists $\boldsymbol\mu=\{\mu_t\}_{t\in[0,r(\sigma_\Phi(\mu))]}\in\mathcal A^p_{[0,r(\sigma_\Phi(\mu))]}(\mu)$ such that 
\[\inf_{t\in[0,r(\sigma_\Phi(\mu))]}\{\sigma_\Phi(\mu_t)-\sigma_\Phi(\mu)\}\le 2\, Q(r(\sigma_\Phi(\mu)),\sigma_\Phi(\mu)).\]
\end{enumerate}
\end{definition}

\begin{remark}
Roughly speaking, $(r,Q)$-attainability expresses a relation between the variation of the distance (or a related positive function vanishing only on the target) along a particular admissible 
trajectory in a time interval, and the size of the time interval itself. The integrability condition asks that the approaching speed, which can be seen as the quotient 
between the variation of the distance and the time needed to realize it, is sufficiently high to ensure that the target will be reached in finite time.
Finite-dimensional examples of similar constructions can be found e.g. in \cite{LM}, while an example in this setting with no interactions can be found in \cite{Cav}.
\end{remark}

With the notations of Definitions \ref{def:Lsigma}, \ref{def:Sattain} we state the following results of this section.
\begin{proposition}\label{prop:Tattain}
Assume Hypothesis \ref{HP} for $F$ and that the generalized target $\tilde S^{\Phi}_p$ is $(r,Q)$-attainable.
Then 
\[\boldsymbol T(\mu):=\int_0^{\sigma_{\Phi}(\mu)}\dfrac{r(q)\,dq}{|Q(r(q),q)|}\ge\tilde T_p(\mu).\]
\end{proposition}
\begin{proof}
Define sequences $\{\mu^{(i)}\}_{i\in\mathbb N}\subseteq \mathscr P(\mathbb R^d)$, $\{\sigma_i\}_{i\in\mathbb N}$, $\{t_i\}_{i\in\mathbb N}\subseteq [0,1]$ as follows.
Set $\mu^{(0)}=\mu$. Suppose to have defined $\mu^{(i)}$, then define $\sigma_i=\sigma_\Phi(\mu^{(i)})$. 
We notice that, by assumption, if $\mu^{(i)}\notin\tilde S^{\Phi}_p$ we have $\sigma_i>0$, and so $Q(r(\sigma_i),\sigma_i)<0$.

By property (4) in Definition \ref{def:Sattain}, if $\mu^{(i)}\notin\tilde S^\Phi_p$ there exists $\boldsymbol\mu^{(i)}=\{\mu^{(i)}_t\}_{t\in[0,r(\sigma_i)]}\in\mathcal A^p_{[0,r(\sigma_i)]}(\mu^{(i)})$ such that 
\[\inf_{t\in[0,r(\sigma_i)]}\{\sigma_\Phi(\mu^{(i)}_t)-\sigma_i\}\le 2\, Q(r(\sigma_i),\sigma_i).\]
Thus, for any $\varepsilon>0$ there exists $t_i^\varepsilon\in[0,r(\sigma_i)]$ such that
\[\sigma_\Phi(\mu^{(i)}_{t_i^\varepsilon})-\sigma_i\le 2\, Q(r(\sigma_i),\sigma_i)+\varepsilon.\]
Notice that, if we choose $\varepsilon$ sufficiently small, in particular $0<\varepsilon<-2\,Q(r(\sigma_i),\sigma_i)$, then $t_i^\varepsilon\neq0$.
We thus fix $\hat\varepsilon(i)=-Q(r(\sigma_i),\sigma_i)$ and set $t_i=t_i^{\hat\varepsilon(i)}>0$ and $\mu^{(i+1)}=\mu^{(i)}_{t_i}$.

While, if $\mu^{(i)}\in\tilde S^\Phi_p$, then we set $t_i=0$ and $\mu^{(i+1)}=\mu^{(i)}_{t_i}=\mu^{(i)}$.

Thus, together with property (1) in Definition \ref{def:Sattain}, this implies that $\mu^{(i)}\notin\tilde S^{\Phi}_p$ if and only if $\sigma_i,\,t_i>0$.

Notice that $\sigma_{i}\ge 0$ for all $i\in\mathbb N$, moreover, if $\sigma_i=0$ then $\sigma_{m}=t_m=0$ for all $m\ge i$.

\medskip

For every $i\in\mathbb N$ such that $\sigma_i\ne 0$ we have
\begin{equation}\label{eq:attsigma}
\sigma_{i+1}-\sigma_i\le 2\,Q(r(\sigma_i),\sigma_i)+\hat\varepsilon(i)=Q(r(\sigma_i),\sigma_i)<0,
\end{equation}
by property (2) in Definition \ref{def:Sattain}.
Thus the sequence $\{\sigma_i\}_{i\in\mathbb N}$ is decreasing and bounded from below, and so it has a limit $\sigma_{\infty}\ge 0$.
If $\sigma_i=0$ for some $i\in\mathbb N$ then $\sigma_{\infty}=0$. If $\sigma_i\ne 0$ for all $i\in\mathbb N$, by passing to the limit in \eqref{eq:attsigma} we get $Q(r(\sigma_{\infty}),\sigma_\infty)=0$, by continuity of $Q(\cdot,\cdot)$ and $r(\cdot)$, which implies $\sigma_{\infty}=0$.

\medskip

We have
\[\boldsymbol T(\mu)\ge \sum_{\substack{i\in\mathbb N\\ \sigma_i\ne 0}}\dfrac{r(\sigma_i)(\sigma_{i}-\sigma_{i+1})}{|Q(r(\sigma_i),\sigma_i)|}\ge \sum_{i\in\mathbb N}t_i.\]
To conclude the proof, we consider two cases
\begin{itemize}
\item  assume that $\sigma_i\ne 0$ for all $i\in\mathbb N$. Then for any $i\in\mathbb N$ there exists an admissible trajectory $\boldsymbol\mu^{(\infty)}=\{\mu^{(\infty)}_t\}_{[0,\boldsymbol T]}$ starting from $\mu$ and coinciding with $\boldsymbol\mu^{(i)}$ on $[t_{i-1},t_i]$. 
In particular, $\mu^{(\infty)}_{\sum t_i}\in \tilde S^{\Phi}_p$ since $\sigma_{\infty}=0$, and so $\boldsymbol T(\mu)\ge \tilde T_p(\mu)$.
\item let $\hat\imath$ the minimum of the set $\{i\in\mathbb N:\,t_i=0\}$. Then there exists an admissible trajectory $\boldsymbol{\hat\mu}^{(\hat\imath)}=\{{\hat\mu}^{(\hat\imath)}_t\}_{[0,\boldsymbol T]}$ starting from $\mu$ and coinciding with $\boldsymbol\mu^{(i)}$ on $[t_{i-1},t_i]$, for all $i\ge1$. 
In particular, ${\hat\mu}^{(\hat\imath)}_{\sum_{i=1}^{\hat\imath-1} t_i}\in \tilde S^{\Phi}_p$, and so 
\[\boldsymbol T(\mu)\ge \sum_{i=1}^{\infty} t_i=\sum_{i=1}^{\hat\imath-1} t_i\ge \tilde T_p(\mu).\]
\end{itemize}
Thus in both cases we have $\boldsymbol T(\mu)\ge \tilde T_p(\mu)$, which concludes the proof.
\end{proof}

\begin{theorem}[Sufficient condition for \textbf{\textup{(STLA)}}]\label{thm:suffSTLA}
Assume Hypothesis \ref{HP} for $F$ and that the generalized target $\tilde S^{\Phi}_p$ is $(r,Q)$-attainable. Assume that there exists $C>0$ and an open set $U\subseteq\mathscr P_p(\mathbb R^d)$ such that $U\supseteq\tilde S^\Phi_p$ and
$\sigma_{\Phi}(\mu)\le C$ for all $\mu\in U$. Then \textbf{\textup{(STLA)}} holds.
\end{theorem}
\begin{proof}
Fix $\varepsilon>0$. Since $\max\{\sigma_{\Phi}(\mu),0\}\le C$ in a neighborhood of $\tilde S^\Phi_p$, we have that the convex function 
$\mu\mapsto\max\{\sigma_{\Phi}(\mu),0\}$ is continuous in a neighborhood of $\tilde S^\Phi_p$ and vanishes exactly on $\tilde S^\Phi_p$.
Thus for any $\varepsilon>0$ there exists $\rho,\delta>0$ such that if $d_{\tilde S^\Phi_p}(\mu)\le \delta$ we have $\sigma_{\Phi}(\mu)\le \rho$ and
\[\varepsilon>\int_0^{\rho}\dfrac{r(q)\,dq}{|Q(r(q),q)|}\ge\tilde T^{\Phi}_p(\mu),\]
recalling that by the integrability assumption in item (3) in Definition \ref{def:Sattain}, the map
\[\rho\mapsto\int_0^{\rho}\dfrac{r(q)\,dq}{|Q(r(q),q)|}\]
is continuous.
\end{proof}

In conclusion, in order to check the $(r,Q)$-attainability of a set from the data of the problem, the following result may serve the purpose.

\begin{corollary}\label{cor:pre-example}
Given $\alpha\ge 0$, an interval $I\subseteq\mathbb R$, $\gamma\in AC(I;\mathbb R^d)$ and $v:\mathbb R^d\to\mathbb R^d$, we define
\[\Delta^v_{\alpha,\gamma}(t):=\left|\dfrac{\gamma(t)-\gamma(0)}{t^{1+\alpha}}-v(\gamma(0))\right|.\]
Let $\mathcal D\subseteq\mathscr P_2(\mathbb R^d)$ and assume that there exist constants $C_\phi\ge 0$, $\alpha,\beta,K>0$ such that, by defining for any $\mu\in\mathcal D$
\begin{equation*}
t_\mu:=\min\left\{1,\dfrac{1}{2L},\sigma_\Phi^{1/\beta}(\mu)\right\} \textrm{ and } I_\mu:=[0,t_\mu],
\end{equation*}
we have
\begin{enumerate}
\item[a.)] $\Phi:=\{\phi\}$, where $\phi$ is semiconcave with constant $C_\phi$;
\item[b.)] for all $\mu\in\mathcal D\setminus \tilde S^\Phi_2$ there exist functions $v_\mu,\xi_\mu\in L^2_\mu(\mathbb R^d;\mathbb R^d)$, 
$\boldsymbol\eta\in\mathscr P(\mathbb R^d\times\Gamma_{I_\mu})$, and constants $C_{2,\mu},C_{3,\mu},C_{4,\mu}>0$
satisfying
\begin{itemize}
\item $0\le \alpha<\beta-1$;
\item $\displaystyle\boldsymbol\mu=\{e_t\sharp\boldsymbol\eta\}_{t\in I_\mu}\in\mathcal A_{I_\mu}(\mu)$, with $e_{t_\mu}\sharp\boldsymbol\eta\in\mathcal D$;
\item $\xi_\mu(x)\in \partial^P\phi(x)$ for $\mu$-a.e. $x\in\mathbb R^d$;
\item $\displaystyle\int_{\mathbb R^d}\langle \xi_\mu(x),v_\mu(x)\rangle\,d\mu(x)\le -C_{2,\mu}<0$;
\item $\left(\displaystyle\int_{\mathbb R^d\times\Gamma_{I_\mu}} |\Delta^{v_\mu}_{\alpha,\gamma}(t_\mu)|^2\,d\boldsymbol\eta(x,\gamma)\right)^{1/2}\le C_{3,\mu}t_\mu$;
\item $\|v_\mu\|_{L^2_\mu}\le C_{4,\mu}$;
\item $\left(-C_{2,\mu}+C_{3,\mu}\|\xi_\mu\|_{L^2_\mu}t_\mu+2 C_\phi(C_{3,\mu}^2 t_\mu^2+C_{4,\mu}^2)t_\mu^{\alpha+1}\right)\le -2K\cdot t_\mu$.
\end{itemize}
\end{enumerate}
Then \textbf{\textup{(STLA)}} holds in $\mathcal D$ and for all $\mu\in\mathcal D$ we have 
\[\tilde T^\Phi_2(\mu)\le \begin{cases}
\dfrac{\beta\sigma^{\frac{\beta-\alpha-1}{\beta}}_\Phi(\mu)}{K(\beta-\alpha-1)},&\textrm{ if }\sigma_\Phi(\mu)\le \min\{1,(2L)^{-\beta}\}\\ \\
\dfrac{\beta\,(2L)^{-\beta+\alpha+1}}{K(\beta-\alpha-1)}+\dfrac{1}{K}(2L)^{\alpha+1}\,(\sigma_\Phi(\mu)-(2L)^{-\beta}),&\textrm{ if }\sigma_\Phi(\mu)\ge (2L)^{-\beta}=\min\{1,(2L)^{-\beta}\}\\ \\
\dfrac{\beta}{K(\beta-\alpha-1)}+\dfrac{1}{K}(\sigma_\Phi(\mu)-1),&\textrm{ if }\sigma_\Phi(\mu)\ge 1=\min\{1,(2L)^{-\beta}\}.\end{cases}\]
\end{corollary}
\begin{proof}
Indeed, for $\boldsymbol\eta$-a.e. $(x,\gamma)\in\mathbb R^d\times\Gamma_{I_\mu}$ we have
\begin{align*}
\phi(\gamma(t_\mu))-\phi(\gamma(0))
\le&\langle \xi_\mu(\gamma(0)), \gamma(t_\mu)-\gamma(0)\rangle+C_\phi|\gamma(t_\mu)-\gamma(0)|^2\\
\le&t_\mu^{\alpha+1}\langle \xi_\mu(\gamma(0)),v_\mu(\gamma(0))\rangle+t_\mu^{\alpha+1}|\xi_\mu(\gamma(0))|\Delta^{v_\mu}_{\alpha,\gamma}(t_\mu)+\\
&+C_\phi t_\mu^{2(\alpha+1)}\left(\Delta^{v_\mu}_{\alpha,\gamma}(t_\mu)+|v_\mu(\gamma(0))|\right)^2.
\end{align*}
Integrating w.r.t. $\boldsymbol\eta$ and using H\"older's inequality yields 
\begin{equation}\label{eq:condLsigma}
\begin{split}
\sigma_{\Phi}(\mu_{t_\mu})-&\sigma_{\Phi}(\mu)\le\\
\le&-C_{2,\mu} t_\mu^{\alpha+1}+\int_{\mathbb R^d\times\Gamma_{I_\mu}}|\xi_\mu(x)|\cdot t_\mu^{\alpha+1}\Delta^{v_\mu}_{\alpha,\gamma}(t_\mu)\,d\boldsymbol\eta(x,\gamma)+2 C_\phi(C_{3,\mu}^2 t_\mu^2+C_{4,\mu}^2)t_\mu^{2(\alpha+1)}\\
\le&-C_{2,\mu} t_\mu^{\alpha+1}+t_\mu^{\alpha+2}\|\xi_\mu\|_{L^2_\mu}\cdot C_{3,\mu}+2 C_\phi(C_{3,\mu}^2 t_\mu^2+C_{4,\mu}^2)t_\mu^{2(\alpha+1)}\\
\le&t_\mu^{\alpha+1}\left(-C_{2,\mu}+C_{3,\mu}\|\xi_\mu\|_{L^2_\mu}t_\mu+2 C_\phi(C_{3,\mu}^2 t_\mu^2+C_{4,\mu}^2)t_\mu^{\alpha+1}\right)\\
\le&- 2K\cdot t_\mu^{\alpha+1}\cdot  t_\mu= -2K\,t_\mu^{\alpha+2}.
\end{split}
\end{equation}
Choose 
\begin{align*}r(q):=\min\left\{1,\frac{1}{2L},q^{1/\beta}\right\},&&Q(t,q):=-K\cdot t^{\alpha+1}\cdot r(q).\end{align*}
In particular, we have $r(q)=0$ if and only if $q=0$, $Q(r(q),q)<0$ if $q\ne 0$,
\[\dfrac{r(q)}{|Q(r(q),q)|}=\dfrac{1}{K\min\left\{1,\frac{1}{(2L)^{\alpha+1}},q^{\frac{\alpha+1}{\beta}}\right\}}=\dfrac{1}{K}\max\left\{1,(2L)^{\alpha+1},q^{-\frac{\alpha+1}{\beta}}\right\},\]
which is a decreasing integrable function of $q$. Furthermore, we notice that by definition $t_\mu=r(\sigma_\Phi(\mu))$ and $Q(r(\sigma_\Phi(\mu)),\sigma_\Phi(\mu))=-K\,t_\mu^{\alpha+2}$. Thus, from \eqref{eq:condLsigma} we get
\[\sigma_{\Phi}(\mu_{r(\sigma_\Phi(\mu))})-\sigma_{\Phi}(\mu)\le2 Q(r(\sigma_\Phi(\mu)),\sigma_\Phi(\mu)),\]
and so we showed that $\tilde S^\Phi_2$ is $(r,Q)$-attainable.
The result now follows from Theorem \ref{thm:suffSTLA} and Proposition \ref{prop:Tattain}.
\end{proof}

\section{A brief comparison with classical attainability}\label{sec:final}
As reported in p. 352 \cite{CD} and at the beginning of Sec. 6.1 in \cite{Car}, we recall that if $(\Omega,\mathcal B,\mathbb P)$ is a sufficiently 
``rich'' probability space, i.e., $\Omega$ is a complete separable metric space, $\mathcal B$ is the Borel $\sigma$-algebra on $\Omega$, and $\mathbb P$ 
is an atomless Borel probability measure, given any $\mu_1,\mu_2\in\mathscr P_2(\mathbb R^d)$
there exist $X_1,X_2\in L^p_{\mathbb P}(\Omega)$ such that $\mu_i=X_i\sharp\mathbb P$, $i=1,2$, and $W_p(\mu_1,\mu_2)=\|X_1-X_2\|_{L^p_{\mathbb P}}$. 
For instance, we can take $\Omega=[0,1]$ endowed with the restriction of the Lebesgue measure to $[0,1]$.
This allows to use the well-known differential structure on $L^p_{\mathbb P}(\Omega)$ (the case $p=2$ is the most common, in particular in the context of \emph{mean field games}) 
in order to formulate the problem and possible derive finer properties of regularity for the solutions, by relying for instance on the theory of viscosity solution in infinite-dimensional 
Banach spaces.
For instance, in this setting (more oriented to a stochastic process interpretation) the representation of Remark \ref{rmk:SPLp} can be expressed as follows.

\begin{corollary}[Stochastic SP]\label{cor:stochSP}
Let $(\Omega,\mathcal F,\mathbb P)$ be a reference probability space, where $\Omega$ is a Polish space and $\mathbb P\in\mathscr P(\Omega)$ an atomless measure. Then, $\boldsymbol\mu=\{\mu_t\}_{t\in [0,T]}\subseteq\mathscr P_2(\mathbb R^d)$ is an admissible trajectory if and only if there exists a stochastic process $X=X(\cdot)$ with
\[\Omega\ni\omega\mapsto X(\cdot):=X(\cdot,\omega)\in AC(0,T)\cap C([0,T];\mathbb R^d)\]
such that 
\begin{itemize}
\item $\mu_t=X_t\sharp\mathbb P$ for all $t\in[0,T]$;
\item $\dot X(t)\in F(\mu_t,X(t))$ for a.e. $t\in[0,T]$.
\end{itemize}
\end{corollary}
\begin{proof}
The proof is a consequence of Proposition \ref{prop:SP}, it sufficies to give the relation between the measure $\boldsymbol\eta$ of Proposition \ref{prop:SP} and the stochastic process $X$. By e.g. Lemma 5.29 in \cite{CD}, there exists a Borel map $\mathscr V:\Omega\to \mathbb R^d\times \Gamma_T$ such that $\boldsymbol\eta=\mathscr V\sharp\mathbb P$. We can thus conclude by setting $X_t:=e_t\circ\mathscr V$ for all $t\in[0,T]$, indeed we have $\mu_t=e_t\sharp\boldsymbol\eta=X_t\sharp\mathbb P$.
\end{proof}

Here we will not follow this approach, since it is not in the purposes of the present paper to enter into this theory. However we actually implemented this more stochastic approach in a similar context in the preprint \cite{CMQ}. Indeed, there our interest is the study of a viability problem in the probability measure space $(\mathscr P_2(\mathbb R^d),W_2)$  by means of a suitable lifted Hamiltonian in  $L^2(\Omega)$.
For completeness, we mention that a theory of well-posedness for Hamilton-Jacobi equations in metric spaces has been introduced and developed for instance by \cite{AF14,AGa,FK06, FN12,GT2019,GS}.

\medskip

In this section, we provide viscosity results related to our study with the purpose to compare the concept of $(r,Q)$-attainability given in Definition \ref{def:Sattain} with the classical one provided by \cite{KQ} in the finite dimensional framework.
For this sake, we first study an Hamilton-Jacobi-Bellman equation associated with our time-optimal control problem in a suitable viscosity sense. In particular, we prove that the minimum time function is a viscosity supersolution of an HJB equation similarly to what occurs in the finite dimensional case.

\begin{definition}[Superdifferential]
Let $1<p<+\infty$, $U:\mathscr P_p(\mathbb R^d)\to\mathbb R$, $\mu\in\mathscr P_p(\mathbb R^d)$ and let $p'$ be the conjugate exponent of $p$. We say that $q\in L^{p'}(\mathbb R^d)$ belongs to the viscosity \emph{superdifferential} of $U$ at $\mu$, and we write $q\in D^+U(\mu)$, if for all $\nu\in\mathscr P_p(\mathbb R^d)$ and all $\pi\in\Pi(\mu,\nu)$ we have
\[U(\nu)\le U(\mu)+\int_{\mathbb R^d\times\mathbb R^d}\langle q(x),y-x\rangle\,d\pi(x,y)+o\left(\left[\int_{\mathbb R^d\times\mathbb R^d}|x-y|^p\,d\pi(x,y)\right]^{1/p}\right).\]
Similarly, the set of viscosity \emph{subdifferentials} of $U$ at $\mu$ is defined by $D^-U(\mu)=-D^+(-U)(\mu)$. 
\end{definition}
\begin{definition}[Viscosity solution]
Let $1<p<+\infty$, and $U:\mathscr P_p(\mathbb R^d)\to\mathbb R$. Let $\mathscr H(\mu,q)$ be defined for any $\mu\in\mathscr P_p(\mathbb R^d)$ and $q\in L^{p'}_\mu(\mathbb R^d)$. We say that
\begin{itemize}
\item $U$ is a \emph{viscosity subsolution} of $\mathscr H(\mu, DU(\mu))=0$ if $U$ is u.s.c. and $\mathscr H(\mu,q_\mu)\le 0$ for all $q_\mu\in D^+U(\mu)$ and $\mu\in\mathscr P_p(\mathbb R^d)$;
\item $U$ is a \emph{viscosity supersolution} of $\mathscr H(\mu, DU(\mu))=0$ if $U$ is l.s.c. and $\mathscr H(\mu,p_\mu)\ge 0$ for all $p_\mu\in D^-U(\mu)$ and $\mu\in\mathscr P_p(\mathbb R^d)$;
\item $U$ is a \emph{viscosity solution} of $\mathscr H(\mu, DU(\mu))=0$ if it is both a super and a subsolution.
\end{itemize}
\end{definition}

In the following, we prove that the minimum time function is a viscosity solution of an Hamilton-Jacobi-Bellman equation with the Hamiltonian $\mathscr H$ defined as follows
\begin{equation}\label{eq:defHamilt}
\mathscr H(\mu,q(\cdot)):=-1-\inf_{\substack{v\in L^p_{\mu}\\ v(x)\in F(\mu,x)}}\langle q(\cdot), v(\cdot)\rangle_{L^{p'},L^p}
\end{equation}
for $\mu\in\mathscr P_p(\mathbb R^d)$, $q(\cdot)\in L^{p'}_{\mu}(\mathbb R^d)$
where $p'$ is the conjugate exponent of $1<p<+\infty$.

\begin{proposition}\label{prop:HJBT}
Let $1<p<+\infty$ and  $\tilde S_p$ be 
a generalized target. Assume Hypothesis \ref{HP} for $F$, and that \textbf{\textup{(STLA)}} holds for the system. Then the minimum time function $\tilde T_p$ is a viscosity solution of the HJB equation $\mathscr H(\mu, D\tilde T_p(\mu))=0$, with Hamiltonian $\mathscr H$ defined in \eqref{eq:defHamilt}.
\end{proposition}
\begin{proof}
By \textbf{\textup{(STLA)}} assumption and Proposition \ref{prop:STLAcontT}, we get the continuity of $\tilde T_p$.

Let $\mu\in\mathscr P_p(\mathbb R^d)$.
Given a function $v_\mu\in L^p(\mu)$ with $v_\mu(x)\in F(\mu,x)$ for $\mu$-a.e. $x$, there exists an admissible trajectory $\boldsymbol\mu=\{\mu_t\}_{t\in [0,T]}$ represented by $\boldsymbol\eta$
satisfying Lemma \ref{lemma:invelset}(1).
According to the Dynamic Programming Principle (Proposition \ref{prop:DPP}(3)), for all $q\in D^+ \tilde T_p(\mu)$  and for all $\pi_t\in \Pi(\mu,\mu_t)$
\begin{align*}
0\le&\dfrac{\tilde T_p(\mu_t)-\tilde T_p(\mu)+t}{t}\\
\le&1+\dfrac{1}{t}\iint_{\mathbb R^d\times\mathbb R^d}\langle q(x),y-x\rangle d\pi_t(x,y)+\dfrac 1t o\left(\left[\iint_{\mathbb R^d\times\mathbb R^d}|x-y|^p\,d\pi_t(x,y)\right]^{1/p}\right)\\
\le&1+\iint_{\mathbb R^d\times\Gamma_T} \langle q(x),\dfrac{\gamma(t)-\gamma(0)}{t}\rangle d\boldsymbol\eta(x,\gamma)+\dfrac{1}{t}o\left(\left[\iint_{\mathbb R^d\times\Gamma_T}|\gamma(t)-\gamma(0)|^p\,d\boldsymbol\eta(x,\gamma)\right]^{1/p}\right)\\
\end{align*}
where we chose $\pi_t=(e_t,e_0)\sharp \boldsymbol\eta$ in the last line. By letting $t\to 0^+$, Lemma \ref{lemma:invelset}(1) yields
\begin{align*}
0\le&1+\int_{\mathbb R^d} \langle q(x),v_\mu(x)\rangle \,d\mu(x).
\end{align*}
By taking the infimum on $v_\mu\in L^p(\mu)$ s.t. $v_\mu(x)\in F(\mu,x)$ for $\mu$-a.e. $x\in\mathbb R^d$, we have for all $\pi^{(\varepsilon)}_t\in \Pi(\mu,\mu_t)$
\[\mathscr H(\mu, q(\mu))\le 0,\]
thus $\tilde T_p(\cdot)$ is a viscosity subsolution of $\mathscr H(\mu,D\tilde T_p(\mu))=0$.

\medskip

On the other hand, from the Dynamic Programming Principle,  for any $\varepsilon>0$ we also get the existence of an admissible trajectory $\boldsymbol\mu^{(\varepsilon)}=\{\mu^{(\varepsilon)}_t\}_{t\in [0,T]}$, 
represented by $\boldsymbol\eta^{(\varepsilon)}$ satisfying Lemma \ref{lemma:invelset}(2), such that for all $p\in D^- \tilde T_p(\mu)$ and for all $\pi^{(\varepsilon)}_t\in \Pi(\mu,\mu^{(\varepsilon)}_t)$
\begin{align*}
\varepsilon\ge&\dfrac{\tilde T_p(\mu^{(\varepsilon)}_t)-\tilde T_p(\mu)+t}{t}\\
\ge& 1+\dfrac{1}{t}\int_{\mathbb R^d}\langle p(x),y-x\rangle d\pi^{(\varepsilon)}_t(x,y)+\dfrac 1to\left(\left[\iint_{\mathbb R^d\times\mathbb R^d}|x-y|^p\,d\pi^{(\varepsilon)}_t(x,y)\right]^{1/p}\right)\\
\ge& 1+\int_{\mathbb R^d\times\Gamma_T}\langle p(x),\dfrac{\gamma(t)-\gamma(0)}{t}\rangle d\boldsymbol\eta^{(\varepsilon)}(x,\gamma)+\dfrac 1t o\left(\left[\iint_{\mathbb R^d\times\mathbb R^d}|\gamma(t)-\gamma(0)|^p\,d\boldsymbol\eta^{(\varepsilon)}(x,\gamma)\right]^{1/p}\right)\\
\end{align*}
where we chose $\pi^{(\varepsilon)}_t=(e_t,e_0)\sharp \boldsymbol\eta^{(\varepsilon)}$ in the last line.
Consider the disintegration of $\boldsymbol\eta^{(\varepsilon)}$ with respect to the evaluation operator at time $t=0$, i.e. $\boldsymbol\eta^{(\varepsilon)}=\mu\otimes \eta^{(\varepsilon)}_x$.
By Filippov's measurable selection Theorem (see \cite[Theorem 8.2.10, Corollary 8.2.13]{AuF})
there exists $w_t\in L^p_{\mu}(\mathbb R^d)$ such that $w_t(x)\in F(\mu,x)$ for $\mu$-a.e. $x\in\mathbb R^d$ and
\begin{align*}
\int_{\mathbb R^d}\left|w_t(x)-\int_{e_0^{-1}(x)}\dfrac{\gamma(t)-\gamma(0)}{t}\,d\eta^{(\varepsilon)}_x\right|^p\,d\mu(x)=&\int_{\mathbb R^d}d^p_{F(\mu,x)}\left(\int_{e_0^{-1}(x)}\dfrac{\gamma(t)-\gamma(0)}{t}\,d\eta^{(\varepsilon)}_x\right)\,d\mu(x)\\
\le&\int_{\mathbb R^d\times\Gamma_T}d^p_{F(\mu,x)}\left(\dfrac{\gamma(t)-\gamma(0)}{t}\right)\,d\boldsymbol\eta^{(\varepsilon)}(x,\gamma),
\end{align*}
where we used Jensen's inequality in the last step. The last term vanishes as $t\to 0^+$ by Lemma  \ref{lemma:invelset}(2), and therefore for $t$ sufficiently small
\begin{align*}
\varepsilon\ge&1-\varepsilon\|p\|_{L^{p'}_{\mu}}+\int_{\mathbb R^d}\langle p(x),w_t(x)\rangle d\mu(x)+\dfrac 1t o\left(t\cdot(\varepsilon+\|w_t\|_{L^p_{\mu}})\right)\\
\ge& -\mathscr H(\mu,p(\cdot))-\varepsilon\|p\|_{L^{p'}_{\mu}}+\dfrac 1t o\left(t\cdot(\varepsilon+\|w_t\|_{L^p_{\mu}})\right)
\end{align*}
Recalling that for $\mu$-a.e. $x\in\mathbb R^d$ we have
\[w_t(x)\in F(\mu,x)\subseteq F(\delta_0,0)+\mathrm{Lip}(F)(W_p(\mu,\delta_0)+|x|),\]
we have that $\|w_t\|_{L^p_\mu}$ is bounded uniformly w.r.t. $t,\varepsilon$, and so by letting $t\to 0^+$ and $\varepsilon\to 0^+$ we get 
\[\mathscr H(\mu,p(\cdot))\ge0.\]
Thus $\tilde T_p(\cdot)$ is also a viscosity supersolution of $\mathscr H(\mu,D\tilde T_p(\mu))=0$.
\end{proof}

\begin{remark}
It seems natural that the combined use of \textbf{\textup{(STLA)}} condition, of Gr\"onwall estimate (see Theorem \ref{thm:filippov}) together with a semiconcavity assumption for the distance to the target
can be used to derive a semiconcavity estimate for the generalized minimum time function, pretty much as in the finite-dimensional case (see e.g. \cite{CS0}).
This topic will be subject of future investigation.
\end{remark}

\medskip

To conclude, we want to perform a comparison between the concept of $(r,Q)$-attainability of Definition \ref{def:Sattain} and the results obtained in \cite{KQ}.
\begin{proposition}
Let $1<p<+\infty$. Assume Hypothesis \ref{HP} for $F$ and that the generalized target $\tilde S_p^\Phi$ is $(r,Q)$-attainable, with $Q=Q(t,s)$ differentiable and such that 
$\partial_t Q(0,\cdot)<0$. Then
$\sigma_\Phi(\cdot)$, defined in Definition \ref{def:Lsigma}, is a viscosity supersolution of 
$1+\mathscr H(\mu, D\sigma_\Phi(\mu))=0$, in particular
\begin{equation}\label{eq:KQ}
\inf_{\substack{v\in L^p_{\mu}\\ v(x)\in F(\mu,x)}}\langle p(x), v_\mu(x)\rangle < 0,
\end{equation}
for all $p(\cdot)\in D^-\sigma_{\Phi}(\mu)$, $\mu\in\mathscr P_p(\mathbb R^2)\setminus\tilde S^\Phi_p$.
\end{proposition}
\begin{proof}
Let $\mu\in\mathscr P_p(\mathbb R^d)\setminus\tilde S^\Phi_p$ and $\boldsymbol\mu\in \mathcal A^p_{[0,t]}(\mu)$ be an admissible trajectory represented by $\boldsymbol\eta$ according to Proposition \ref{prop:SP} and Remark \ref{rmk:SPLp}, with $0<t<r(\sigma_\Phi(q))$.
We notice that, for all $p(\cdot)\in D^-\sigma_{\Phi}(\mu)$ we have
\[\sigma_\Phi(\mu_{t})-\sigma_\Phi(\mu)\ge \int_{\mathbb R^d\times\Gamma_{[0,t]}} \langle p(x),\gamma(t)-\gamma(0)\rangle\,d\boldsymbol\eta(x,\gamma)+o\left(\|e_t-e_0\|_{L^p_{\boldsymbol\eta}}\right).\]
In particular, according to $(r,Q)$-attainability condition, for every $\varepsilon>0$ and $0<t\le\varepsilon$ we can find $\boldsymbol\mu^{(\varepsilon)}\in \mathcal A^p_{[0,t]}(\mu)$ represented by $\boldsymbol\eta^{(\varepsilon)}$
such that
\[\frac{2Q(t,\sigma_\Phi(\mu))}{t}+\varepsilon\ge \frac{1}{t}\iint_{\mathbb R^d\times\Gamma_{[0,t]}} \langle p(x),\gamma(t)-\gamma(0)\rangle\,d\boldsymbol\eta^{(\varepsilon)}(x,\gamma)+\frac{1}{t}o\left(\|e_t-e_0\|_{L^p_{\boldsymbol\eta^{(\varepsilon)}}}\right),\]
where we used the previous estimate and divided by $t>0$.
Recalling the hypothesis on $Q$ and that without loss of generality we can consider $Q(0,\cdot)=0$, then by letting $t\to 0^+$ and $\varepsilon\to 0^+$ we obtain
\[0>2\partial_t Q(0,\sigma_\Phi(\mu))\ge -1-\mathscr H(\mu,p(\cdot)).\]
Here, we used the same arguments used in Proposition \ref{prop:HJBT} when proving that $\tilde T_p$ is a viscosity supersolution of an HJB equation with Hamiltonian $\mathscr H$ defined in \eqref{eq:defHamilt}. We remind that in order to be a supersolution, the continuity of $\tilde T_p$ is not needed, and the l.s.c. is provided by Proposition \ref{prop:DPP}.

So, we got that $\sigma_\Phi(\cdot)$ is a viscosity supersolution of $1+\mathscr H(\mu, D\sigma_\Phi(\mu))=0$, thus the conclusion noting that the expression above can be equivalently written as
\[\inf_{\substack{v\in L^p_{\mu}\\ v(x)\in F(\mu,x)}}\langle p(x), v_\mu(x)\rangle \le 2\partial_t Q(0,\sigma_\Phi(\mu))<0.\]
\end{proof}

\begin{remark}
Notice that the expression obtained above in \eqref{eq:KQ} represents a natural counterpart of formula (27) in \cite{KQ} (and subsequent extensions, see Corollary 3.5 and Remark 3.3 of \cite{KQ}).
Here, the Lipschitz continuity requested in \cite{KQ} is replaced by an integrability assumption of the derivative $\partial_t Q$.
Therefore, under these hypothesis, $(r,Q)$-attainability can be seen as a \emph{sampled} form of the assumption in \cite{KQ}, indeed in our framework the decreasing condition (27) is checked
along admissible trajectories only at time steps of size given by $r(\cdot)$.
\end{remark}

\bigskip

\textbf{Acknowledgements}: G.C. has been partially supported by Cariplo Foundation and Regione Lombardia via project \emph{Variational Evolution Problems and Optimal Transport} and by MIUR PRIN 2015 project \emph{Calculus of Variations}, together with FAR funds of the Department of Mathematics of the University of Pavia.\\
The authors finally thank both the referees for their suggestions which definetely helped us to improve the presentation of the paper and to better frame it in the existing literature.

\end{document}